\documentclass[11pt]{amsart}
\usepackage[mathscr]{eucal}
\usepackage{times}
\usepackage{amsmath,amssymb,latexsym,amscd,amsthm}   
\usepackage{hyperref}
\hypersetup{colorlinks=true,linkcolor=blue,citecolor=red,linktocpage=true}
\usepackage{graphicx}
\usepackage[all,cmtip]{xy}
\usepackage{fancyhdr}
\usepackage{enumitem}
\usepackage{mathrsfs}
\usepackage{upgreek}
\usepackage{xcolor}
\usepackage{textcase}

\DeclareMathOperator{\Mod}{-Mod}

\DeclareMathOperator{\Ann}{Ann}

\DeclareMathOperator{\pt}{pt}

\DeclareMathOperator{\mx}{\it{Max}}
\DeclareMathOperator{\Max}{\it{Max}}

\DeclareMathOperator{\Hom}{Hom}
\DeclareMathOperator{\End}{End}
\theoremstyle{plain}
\newtheorem*{theorem*}{Theorem}
\newtheorem{thm}{Theorem}[section]
\newtheorem{cor}[thm]{Corollary}
\newtheorem{lem}[thm]{Lemma}
\newtheorem{prop}[thm]{Proposition}

\theoremstyle{definition}
\newtheorem{dfn}[thm]{Definition}
\newtheorem{obs}[thm]{Remark}
\newtheorem{ej}[thm]{Example}

\newcommand{\V}{\mathcal{V}}
\newcommand{\U}{\mathcal{U}}

\newcommand{\sm}{\sigma[M]}

\newcommand{\F}{\mathcal{F}}
\newcommand{\ltc}{\mathcal{L}} 

\newcommand\blfootnote[1]{%
  \begingroup
  \renewcommand\thefootnote{}\footnote{#1}%
  \addtocounter{footnote}{-1}%
  \endgroup
}

\title[ Strongly harmonic and Gelfand modules]{ 
On strongly harmonic and Gelfand modules }
\author[Medina,\\ Morales,\\ Sandoval,\\ Zald\'ivar]{Mauricio Medina-B\'arcenas \\  Lorena Morales-Callejas\\ Martha Lizbeth Shaid Sandoval-Miranda \\  \'Angel Zald\'ivar-Corichi }
\dedicatory{Dedicated to the memory of Professor Harold Simmons.}

\begin{document}

\begin{abstract}
We introduce the notions of strongly harmonic and Gelfand module, as a generalization of the well-known ring theoretic case. We prove some properties of these modules and we characterize them via their lattice of submodules and their space of maximal submodules. It is also observed that, under some assumptions, the space of maximal submodules of a strongly harmonic module constitutes a compact Hausdorff space whose frame of open sets is isomorphic to the frame $\Psi(M)$ defined in \cite{medina2017attaching}. Finally, we mention some open questions that arose during this investigation.

\end{abstract}

\maketitle

\section{Introduction}\label{sec1}

The present manuscript can be considered as a natural step of the investigation initiated in \cite{medina2017attaching}.  
In that document, we associated to a module $M$ (satisfying some conditions) two frames,
the frame of \emph{semiprimitive submodules} $SPm(M)$ and the frame $\Psi(M)$ given by \[N\in\Psi(M)\Leftrightarrow(\forall n\in N)[N+\Ann_{M}(Rn)=M].\] 
It is observed that these two frames are spatial and they work as classification objects of the module $M$ \cite[Theorem 3.13, Theorem 5.6]{medina2017attaching}. In fact, we have that $SPm(M)\cong\mathcal{O}(\mx(M))$ (the frame of open sets of the space of maximal submodules of $M$).

For the frame $\Psi(M)$, it seems that its point space $\pt(\Psi(M))$ is hard to describe, and there is not a direct  connection with the frame of semiprimitive submodules.

In the ring-theoretic case, the point space of $\Psi(R)$ can be described for certain classes of rings, the strongly harmonic and Gelfand rings. 
The general definition of strongly harmonic ring was introduced in \cite{koh1972representation}. In that paper is observed that the space of maximal ideals $\mx(R)$ of a strongly harmonic ring $R$ with the hull-kernel topology is a compact Hausdorff space. Later, in \cite{mulvey1979generalisation} Gelfand rings were introduced, and it was proved that for these rings the space of maximal ideals is also compact Hausdorff (it results that any Gelfand ring is strongly harmonic). The importance of these kind of spaces reside in that strongly harmonic rings can be represented as the ring of global sections over compact Hausdorff spaces \cite[Theorem 3.7]{koh1972representation}. In that path, in \cite{borceux2006algebra} the authors (as an example of a more general theory) introduce a representation for rings based on the frame $\Psi(R)$ defined as the set of \emph{pure ideals} (i.e., ideals $I$ such that $R/I$ is a flat right module). In a more particular setting in \cite{borceux1984sheaf} (see \cite{simmonssheaf} for the strongly harmonic case) is observed that the frame $\Psi(R)$ serves as a good space to unify the known representations and they show that, for Gelfand rings the point space of $\Psi(R)$ is homeomorphic to $\mx(R)$ with the hull-kernel topology, equivalently $\Psi(R)\cong\mathcal{O}(\mx(R))$.

Later in \cite{simmonssome}, the author organizes the ring theoretic properties of strongly harmonic rings and Gelfand rings. Following that manuscript, we introduce the notions of \emph{strongly harmonic module}  and \emph{Gelfand module} and we explore the properties of these modules. We study their space of maximal fully invariant submodules $\mx^{fi}(M)$ for strongly harmonic modules (Theorem \ref{dreamtheo})  and $\mx(M)$ for Gelfand modules and we relate those spaces with the point space of $\Psi(M)$. We will make use of latticial and point-free techniques applied to the \emph{idiom} of submodules of a given module $M$. In fact, many of these results were obtained trying to prove Theorem \ref{dreamtheo} as a reminiscence of \cite[Theorem 3.5]{simmons1989compact} and \cite[Corollary 4.7]{paseka1986regular}.

We now give a brief description of the contents in this paper. 
Section \ref{sec2} is the background material needed to make this manuscript as self-contained as possible. In Section \ref{sec3} we introduce the notion of normal idiom. Using the notion of quasi-quantale and of relative spectrum introduced in \cite{medina2015generalization}, given a quasi-quantal $A$ and a subquasi-quantal $B$ satisfying $(\star)$ (Definition \ref{sqqcond}), the space $Spec_B(A)$ (the spectrum of $A$ relative to $B$) is normal if and only if the fixed point defined by the hull-kernel topology is normal (Proposition \ref{pronormal}). This allows us to characterize the frames of semiprime and semiprimitive submodules (resp. ideals) of a module $M$ (resp. of a ring $R$) in terms of the normality of the spaces $Spec(M)$ and $\mx(M)$ (resp. $Spec(R)$ and $\mx(R)$) (Corollaries \ref{normaleq}--\ref{NormalledMaxring}).
Section \ref{sec4} is the main section and is devoted to the study of strongly harmonic modules and their space of maximal submodules. We give some properties of those modules, we show that factoring out with a fully invariant submodule of a strongly harmonic module inherits the property (Proposition \ref{summandfa}), also we prove that direct sums of copies of a strongly harmonic module is strongly harmonic (Proposition \ref{dirsumfa}). It is proved when the condition of normality on the space $Spec(M)$ or on $\Lambda(M)$ characterizes a strongly harmonic module $M$ (Proposition \ref{norstr}  and Theorem \ref{theoremsth}). We make use of the operator $Ler$ introduced in \cite[Section 5]{medina2017attaching} to prove a characterization (Theorem  \ref{dreamtheo}) which will be the key to make a connection with the frame $\Psi(M)$. We see that the frame $\Psi(M)$ is a regular frame (Theorem \ref{dreamcon}) and we  prove that $\pt(\Psi(M))$ is homeomorphic to the space $\mx(M)$ and hence $\Psi(M)\cong \mathcal{O}(\mx(M))$ as frames (Theorem \ref{homeo} and  Corollary \ref{corhomeo}). In Section \ref{sec5}, we present Gelfand modules, we show that for a module $M$ projective in $\sm$, if $M$ is a Gelfand module then $M$ is strongly harmonic; and the converse follows provided that $M$ is quasi-duo (Theorem \ref{SHDuoGelfand}).  In Proposition \ref{summandgel}, it is also observed that for a quasi-projective Gelfand module each factor module is Gelfand. In Theorem \ref{DreamcomG}, it is shown that the operator $Ler$ defines a frame isomorphism between $\Psi(M)$ and $SPm(M)$ for a Gelfand module $M$ provided of additional hypothesis. We present a characterization of Gelfand modules (Theorem \ref{DOS}) in connection the well-known  of Demarco-Orsati-Simmons Theorem \cite{de1971commutative,simmons1980reticulated}. This theorem characterizes commutative Gelfand rings as those rings $R$ such that $\mx(R)$ is a retraction of $Spec(R)$. At the end, some open questions and possible lines to work in are exposed.

\section{preliminaries}\label{sec2}

Throughout this paper $R$ will be an associative ring with identity, not necessarily commutative. The word \emph{ideal} will mean two-sided ideal, unless explicitly stated the side (left or right ideal). All modules are unital and left $R$-modules. Given an $R$-module $M$, a submodule $N$ of $M$ is denoted by $N\leq M$, whereas we write $N<M$ when $N$ is a proper submodule of $M$. Recall that $N\leq M$ is said to be \emph{a fully invariant submodule}, denoted by $N\leq_{fi}M$, if for every endomorphism $f\in \End_R(M)$, it follows that $f(N)\subseteq N.$ Set $\Lambda(M)=\{N\mid N\leq M\},$ and $\Lambda^{fi}(M)=\{N\mid N\leq_{fi}M\}$. Given a module $M$ and a set $X$, the direct sum of copies of $M$ is denoted by $M^{(X)}$, if the set is finite, say $|X|=n$ we write $M^{(n)}$. An $R$-module $N$ is said to be \emph{$M$-generated} if there exists an epimorphism $\rho:M^{(X)}\to N$, and $N$ is \emph{$M$-subgenerated} if $N$ can be embedded into an $M$-generated module. In order to generalize the ring properties to modules we will work in the category $\sm$ for a module $M$. \emph{The category $\sm$} is the full subcategory of $R$-Mod consisting of all $M$-subgenerated modules. It can be seen that if $R=M$ then $\sm=R$-Mod. As the ring $R$ is always projective in $R$-Mod, some projectivity conditions will be needed. Recall that given modules $M$ and $N$, it is said that \emph{$M$ is $N$-projective} if for every epimorphism $\rho\colon N\to X$ and every homomorphism $\alpha:M\to X$ there exists $\beta\colon M\to N$ such that $\rho\beta=\alpha$. The module $M$ is \emph{quasi-projective} if it is $M$-projective. To get deeper results and make a module more tractable some assumptions will be imposed along the paper. Principally, it will be asked for a module $M$ to be \emph{projective in $\sm$} and in some cases that every submodule of $M$ is $M$-generated (\emph{self-generator module}). For undefined notions and general module theory we refer the reader to \cite{lamlectures} and \cite{wisbauerfoundations}.


\begin{dfn}\label{Id} 
An \emph{idiom} $(A,\leq,\bigvee,\wedge, 1, 0)$ is a complete, upper-continuous, modular lattice, that is, $A$ is a complete lattice that satisfies the following distributive laws: 
\vspace{-4pt}\[a\wedge (\bigvee X)=\bigvee\{a\wedge x\mid x\in X\},\]
for all $a\in A$  and $X\subseteq A$ directed; and 
\[a\leq b\Rightarrow (a\vee c)\wedge b=a\vee(c\wedge b)\]
for all $a,b,c\in A$.
\end{dfn} 

Our basic examples of idioms are the complete lattices $\Lambda(M)$ and $\Lambda^{fi}(M)$ for a module $M$.

A distinguish class of idioms, are the distributive ones:
\begin{dfn}
A complete lattice $(A,\leq,\bigvee,\wedge,1,0)$ is a \emph{frame}, if $A$ satisfies
\[a\wedge (\bigvee X)=\bigvee\{a\wedge x\mid x\in X\}\leqno({\rm FDL}),\]
\noindent for all $a\in A$ and $X\subseteq A$ any subset.
\end{dfn}

Of course the prototypical example of a frame comes from topology. Given a topological space $S$ with topology $\mathcal{O}(S)$, it is known that $\mathcal{O}(S)$ is a frame.

The point-free techniques we are interested in are based on the concept of nucleus. We give a quick review of that.

\begin{prop}{\rm(\cite[Lemma 3.1]{simmonsintroduction}).}\label{hullker}
Given any morphism of $\bigvee$-semilattices, $f^{*}\colon A\rightarrow B$ there exists  $f_{*}\colon B\rightarrow  A$ such that 
\[f^{*}(a)\leq b\Leftrightarrow a\leq f_{*}(b),\]
  for each $a\in A$ and $ b\in B.$ 
That is, $f^{*}$ and $f_*$ form an adjunction 
{\vspace{-8pt}}
\[\xymatrix{ A\ar@<-.7ex> @/   _.5pc/[r]_--{f^{*}} & B.\ar@<-.7ex>@/ _.5pc/[l]_--{f_{*}} }\]

In fact, $f_{*}(b)=\bigvee \{x\in A \mid f^{*}(x)\leq b\},\,$ for each $b\in B.$
\end{prop}

This is a particular case of the General Adjoint Functor Theorem. A proof of this can be found in any standard book of category theory, for instance, \cite[Theorem 6.3.10]{leinster2014basic}, and \cite[Lemma 3.1]{simmonsintroduction}.

The reader can see  \cite{simmonsintroduction} and and \cite{rosenthal1990quantales} for more details of all these facts.

\begin{lem}{\rm(\cite[Lemma 3.3]{simmonsintroduction}).}
\label{lema3.3HS}  Let $f^{*}\colon A \to B$ be an arbitrary morphism of $\bigvee$-semilattices, and $f_{*}\colon B\rightarrow  A$ the right adjoint of $f^{*}.$
Then, $\mu:=f_{*}{\circ}f^{*}{\colon} A\to A$ is a closure operation satisfying the following conditions:
\begin{enumerate}
\item[\rm(a)] $x\leq \mu(a)$ if and only if $f^{*}(x)\leq f^{*}(a)$,  for each $x,a\in A,$
\item[\rm(b)] $f^{*}(\mu(a))=f^{*}(a)$, for each $a\in A.$
\end{enumerate}
\end{lem}

\noindent According to the terminology in  \cite[Definition 3.2]{simmonsintroduction}, the function $\mu\colon A\to A$  it is referred to as {\it the kernel of }$f^{*}.$

\begin{dfn}\label{nuc}
Let $A$ be an idiom. A \emph{nucleus} on $A$ is a function $j\colon A\rightarrow A$ such that:
\begin{enumerate}[label={(\alph*)}]
\item $j$ is an inflator.

\item $j$ is idempotent.

\item $j$ is a \emph{prenucleus}, that is, $j(a\wedge b)=j(a)\wedge j(b)$.
\end{enumerate}
\end{dfn}

Applying Proposition \ref{hullker} and Lemma \ref{lema3.3HS}  it gets the following result for idiom morphisms:

\begin{lem}{\rm(\cite[Lemma 3.12]{simmonsintroduction}).}
Let $f^{*}\colon A \to B$ be an arbitrary idiom morphism. Then, the kernel of $f^{\ast}$, namely $\mu\colon A \to A$  is a nucleus.
\end{lem}

As we mentioned before, every topological space $S$ determines a frame, its topology $\mathcal{O}(S)$. This defines a functor from the category of topological spaces to the category of frames $\mathcal{O}(\_)\colon \mathrm{Top}\rightarrow \EuScript{F}\mathrm{rm}$. There exists a functor in the other direction:

\begin{dfn}\label{point}
Let $A$ be a frame. An element $p\in A$ is a \emph{point} or a $\wedge$\emph{-irreducible}
if $p\neq 1$ and $a\wedge b\leq p \Rightarrow a\leq p \text{ or } b\leq p.$
\end{dfn}
Denote by $\pt(A)$ the set of all points of $A$. This set can be endowed with a topology as follows: for each $a\in A$ define

\vspace{4pt}\centerline{\vspace{4pt} $U_{A}(a)=\{p\in\pt(A)\mid a\nleq p\}.$}

The collection $\mathcal{O}\pt(A)=\{U_{A}(a)\mid a\in A\}$ constitutes a topology for $\pt(A)$. We have a frame morphism 
\[U_{A}\colon A\rightarrow \mathcal{O}\pt(A)\]
that determines a nucleus on $A$ by Proposition \ref{hullker}. This nucleus or the adjoint situation is called the \emph{hull-kernel adjunction}. With this, the frame $A$ is \emph{spatial} if $U_{A}$ is an injective morphism (hence an isomorphism). 

It  can be proved that this defines a functor $\pt(\_)\colon \EuScript{F}\mathrm{rm}\rightarrow \mathrm{Top}$ in such way that the pair 
\[\xymatrix{ &\mathrm{ Top}\ar @/^/[r]^{\mathcal{O}(\_)} & \EuScript{F}\mathrm{rm} \ar @/^/[l]^{\pt(\_)}}  \]
forms an adjunction.
For more details, see \cite{johnstone1986stone}, \cite{simmonspoint} and \cite{picado2011frames}, and \cite{simmonsintroduction}.

We need some other point-free structures that generalize idioms and frames.


\begin{dfn}[\cite{medina2015generalization}]
A \emph{quasi-quantale} $A$ is a complete lattice with an associative product $A\times A\to A$ such that for all directed subsets $X,Y\subseteq A$ and $a\in A$:

\centerline{ $\left(\bigvee X\right)a=\bigvee\{xa\mid x\in X\}$}

and

\centerline{$ a\left(\bigvee Y\right)=\bigvee\{ay\mid y\in Y\}.$}
\end{dfn}

\begin{dfn}
A \emph{multiplicative idiom} is an idiom $(A,\leq,\bigvee,\wedge,\cdot)$ with an extra operation compatible with the order in such way $(A,\leq,\bigvee,\cdot)$ is a quasi-quantale.
\end{dfn}

\begin{ej}
For any left $R$-module $M,$ in \cite[Lemma 2.1]{BicanPr}  was defined the product  

\centerline{ \vspace{4pt} $N_{M}L:=\sum\{f(N)\mid f\in\Hom(M,L)\}, $}

\noindent for submodules $N,L\in{\Lambda(M)}$. In \cite[Proposition 1.3]{PepeGab} is proved that 
 \[\left(\sum_I N_i\right) {_ML}=\sum_I\left({N_i}_ML\right),\]
for each family of submodules $\{N_i\}_I$ of $M$ and each $L\leq M$. On the other hand, since $N_M\_$ is a prerradical in $R$-Mod (i.e. a subfunctor of the identity functor),
\[N_M\left(\sum_IL_i\right)=\sum_I(N_ML_i)\]
holds for every directed family $\{L_i\}_I$ of submodules of $M$ and any $N\leq M$. In general this product is not associative, but if $M$ is projective in $\sm$ the product is assocative \cite[Proposition 5.6]{beachy2002m}. Therefore, if $M$ is projective in $\sm$ then $\Lambda(M)$ is a multiplicative idiom. 
\end{ej}

Recently, in \cite[Corollary 1.5]{PepeyJaime} has been shown that for a class of modules called \emph{multiplication modules}, the product $-_M-$ is associative even if the module $M$ is not projective in $\sm$.

\begin{dfn}[\cite{medina2015generalization}, Definition 3.15]
A sub $\bigvee$-semilattice $B$ of a quasi-quantale $A$ is a \emph{subquasi-quantale} if $B$ is a quasi-quantale with the restriction of the product in $A$. 
\end{dfn}

\begin{dfn}\label{sqqcond}
Given a subquasi-quantal $B$ of a quasi-quantale $A$, we will say $B$ satisfies the condition $(\star)$ if $0,1\in B$ and $1b,b1\leq b$ for all $b\in B$.
\end{dfn}

The condition $(\star)$ comes from our canonical example of quasi-quantale $\Lambda(M)$ with $M$ an $R$-module and the canonical subquasi-quantale $\Lambda^{fi}(M)$. Note that $\Lambda^{fi}(M)$ satisfies condition $(\star)$. In general, $\Lambda(M)$ does not satisfies $(\star)$. For example, consider $M=\mathbb{Z}_2\oplus\mathbb{Z}_2$, then $\left(\mathbb{Z}_2\oplus 0\right){_MM}=M\nsubseteq\mathbb{Z}_2\oplus 0$.

\begin{dfn}[\cite{medina2015generalization}, Definition 3.16]
Let $B$ be a subquasi-quantale of a quasi-quantale $A$. An element $1\neq p\in A$ is a \emph{prime element relative to $B$} if whenever $ab\leq p$ with $a,b\in B$ then $a\leq p$ or $b\leq p$. We define the \emph{spectrum relative to $B$ of $A$} as
\[Spec_B(A)=\{p\in A\mid p\;\text{is prime relative to }B\}.\]
In the case $A=B$ this is the usual definition of prime element. We denote the set of prime elements of $A$ by $Spec(A)$.
\end{dfn}

\begin{obs}
In the case $A=\Lambda(M)$ and $B=\Lambda^{fi}(M)$, following \cite{medina2017attaching} we write $LgSpec(M)=Spec_{B}(A)$ and we call it \emph{the large spectrum of $M$} and for $Spec_B(B)$ we just write $Spec(M)$. Note that when $R=M$, $Spec(M)$ is the usual prime spectrum. As it was noticed in \cite[Example 4.14]{medina2015generalization}, if $M$ is quasi-projective then $\mx(M)\subseteq LgSpec(M)$. Moreover, if $M$ is projective in $\sm$, $\emptyset\neq \mx(M)\subseteq LgSpec(M)$ by \cite[22.3]{wisbauerfoundations}.
\end{obs}

\begin{prop}[\cite{medina2015generalization}]\label{t}
Let $B$ be a subquasi-quantale satisfying $(\star)$ of a quasi-quantale $A$. Then $Spec_B(A)$ is a topological space with closed subsets given by

\centerline{\vspace{4pt} $ \V(b)=\{p\in Spec_B(A)\mid b\leq p\}$}

\noindent with $b\in{B}$. 

In dual form, the open subsets are of the form

\vspace{4pt} \centerline{\vspace{4pt} $\U(b)=\{p\in{Spec_B(A)}\mid b\nleq{p}\}$}

\noindent with $b\in{B}$.
\end{prop}

\begin{obs}[\cite{medina2015generalization}]\label{adjunction}
Let $B$ be a subquasi-quantale satisfying $(\star)$ of a quasi-quantale $A$. Let $\mathcal{O}(Spec_B(A))$ be the frame of open subsets of $Spec_B(A)$. We have an adjunction of $\bigvee$-morphisms 
{\vspace{-5pt}}
\[\xymatrix@=20mm{B\ \ \ar@/^/[r]^{\mathcal{U}} & \mathcal{O}(Spec_B(A))\ \ \ar@/^/[l]^{\mathcal{U}_*}}\]
where $\mathcal{U}_*$ is defined as
$\U_*(W)=\bigvee\{b\in{B}\mid\U(b)\subseteq{W}\}.$
The composition $\mu:=\U_*\circ\U$ is a closure operator in $B$. Note that $\U(x)=\U(\mu(x))$ (equivalently, $\V(x)=\V(\mu(x))$) for all $x\in B$.  
\end{obs}

\begin{prop}\label{41}
Let $B$ be a subquasi-quantale satisfying $(\star)$ of a quasi-quantale $A$ and  $\mu=\U_*\circ\U\colon B\to B$ as above. Then, the  following conditions hold.
\begin{enumerate}[label=\emph{(\alph*)}]
\item \cite[Proposition 3.20]{medina2015generalization} For each  $b\in{B},$ $\mu(b)$ is the largest element of $B$ such that 

\vspace{4pt}\centerline{\vspace{4pt} $ \mu(b)\leq\bigwedge\{p\in{Spec_B(A)}\mid p\in\mathcal{V}(b)\}.$}

\item \cite[Theorem 3.21]{medina2015generalization} $\mu$ is a multiplicative nucleus.  
\item \cite[Corollary 3.22]{medina2015generalization} $B_\mu$ is a meet-continuous lattice.
\item \cite[Corollary 3.11]{medina2015generalization} If $B$ satisfies that for any $X\subseteq{B}$ and $a\in{B}$ 

\centerline{\vspace{4pt} $ \left(\bigvee{X}\right)a=\bigvee\{xa\mid x\in{X}\}$}

\noindent then, $B_\mu$ is a frame. 
\end{enumerate}
\end{prop}

\begin{dfn}\label{star}
Let $B$ be a subquasi-quantale satisfying $(\star)$ of a quasi-quantale $A$. We say $A$ satisfies \emph{the p-condition relative to $B$} if for all $b\in B$ there exists $p\in Spec_B(A)$ such that $b\leq p$. If $A=B$ we just say that $A$ satisfies \emph{the p-condition.}
\end{dfn}

\begin{obs}\label{casicoat}
Let $M$ be projective in $\sm$ and set $A=\Lambda(M)$ and $B=\Lambda^{fi}(M)$. Then $A$ satisfies the p-condition relative to $B$. For, let $N\in B$. Since $M$ is projective in $\sm$, $M/N$ is projective in $\sigma[M/N]$. It follows from \cite[22.3]{wisbauerfoundations} that $\mx(M/N)\neq\emptyset$. This implies that there exists $\mathcal{M}\in\mx(M)\subseteq LgSpec(M)$ such that $N\leq\mathcal{M}$.
\end{obs}

\begin{lem}\label{lemmap}
Let $B$ be a subquasi-quantale satisfying $(\star)$ of a quasi-quantale $A$ and $\mu$ the multiplicative nucleus given by the adjoint situation in Remark \ref{adjunction}. Then, for $x,y\in B$, the following conditions hold. 
\begin{enumerate}[label=\emph{(\alph*)}]
\item $x\leq y$ implies $\V(y)\leq \V(x).$
\item $\V(x)=\V(\mu(x))$.
\item $xy\leq \mu(0)$ if and only if $\U(x)\cap \U(y)=\emptyset.$
\item If $x\vee y=1$ then $\V(x)\cap\V(y)=\emptyset.$ If in addition, $A$ satisfies the $p-$condition relative to $B$, then the converse holds. 
\end{enumerate}
\end{lem}
\begin{proof} 
(a) Let $p\in Spec_B(A)$ such that $y\leq p.$ Since $x\leq y,$ it is clear that $x\leq p.$

(b) From the fact that $\mu$ is inflatory and by (a), it follows that $\V(\mu(x))\subseteq \V(x).$ On the other hand, for every $p\in \V(x)$, $\mu(x)\leq p$ by Proposition \ref{41}(a). Then, $\V(x)\subseteq \V(\mu(x)).$

(c) Suppose that $xy\leq \mu(0).$ Then, for every $p\in Spec_B(A),$ $xy\leq p.$ Thus, $Spec_B(A)=\V(xy).$ Therefore, $\emptyset=\U(xy)=\U(x)\cap\U(y).$
Conversely, supposse that $\U(x)\cap\U(y)=\emptyset.$ Then, $Spec_B(A)=\V(xy).$ So, for every $p\in Spec_B(A),$ $xy\leq p$ and hence,  $xy\leq \mu(0).$ 

(d) First, suppose that $x\vee y=1.$ Then, $\V(x\vee y)=\V(1)=\emptyset.$ Thus, $\emptyset=\V(x\vee y)=\V(x)\cap\V(y).$  

On the other hand, suppose that $A$ satisfies the p-condition relative to $B$ and $\V(x)\cap\V(y)=\emptyset$. Hence $\V(x\vee y)=\emptyset.$ 

If $x\vee y\neq 1,$ there exists $1\neq p\in A$ such that $x\vee y\leq p\lneq 1$ implying that $p\in \V(x\vee y)=\emptyset $ which is a contradiction. Thus, $x\vee y=1$.
\end{proof}

\begin{lem}\label{lemaV}  
Let $B$ be a subquasi-quantale satisfying $(\star)$ of a quasi-quantale $A$ and $\mu$ be the multiplicative nucleus given by the adjoint situation on Remark \ref{adjunction}. Consider the following conditions for $x,y\in B$.
\begin{enumerate}[label=\emph{(\alph*)}]
\item $x\vee y=1.$
\item $\mu(x)\vee\mu(y)=1.$
\item $\mu(\mu(x)\vee \mu(y))=1.$
\item $\V(\mu(x))\cap\V(\mu(y))=\emptyset.$
\item $\V(x)\cap\V(y)=\emptyset.$
\end{enumerate}
The implications \emph{(a)}$\Rightarrow$\emph{(b)}$\Rightarrow$\emph{(c)}$\Rightarrow$\emph{(d)}$\Rightarrow$\emph{(e)} hold. If in addition,  $A$ satisfies the p-condition relative to $B$, all these conditions are equivalent.
\end{lem}

\begin{proof}
(a)$\Rightarrow$(b)$\Rightarrow$\textit(c) Follows from the fact that $\mu$ is inflatory.

(c)$\Rightarrow$(d) We have,

\centerline{ $\emptyset=\V(1)=\V(\mu(\mu(x)\vee \mu(y)))=\V(\mu(x)\vee \mu(y))=\V(\mu(x))\cap\V(\mu(y)).$}

(d)$\Rightarrow$(e) It is clear.

Assume $A$ satisfies the p-condition relative to $B$. 

(e)$\Rightarrow$(a) It follows from Lemma \ref{lemmap}.(c).
\end{proof}

Given a complete lattice $\ltc,$ recall that an element $c\in \ltc$ is \textit{compact} if  for every $X\subseteq \ltc$ such that $c\leq \bigvee X,$ there exists a finite subset $F\subseteq X$ satisfying $c\leq \bigvee F.$
Also, recall that a lattice $\ltc$ is said to be a \emph{compact lattice} if and only if $1_\ltc$ is compact in $\ltc$.

In \cite{HMPL2018}, the authors have extensively studied the conditions of compactness in different lattices which have been of interest in the study of module theory.  In particular, \cite[Propositions 4.6,  4.7, and Lemma 4.12]{HMPL2018},  give characterizations of compact elements in   $\Lambda^{fi}(M)$.

\begin{prop}\label{bcompspeccomp}
Let $B$ be a subquasi-quantale satisfying $(\star)$ of a quasi-quantale $A$. Then, $B$ is compact and $A$ satisfies the p-condition relative to $B$ if and only if $Spec_B(A)$ is a compact space. 
\end{prop}

\begin{proof}
Let $\{\U(b_i)\}_I$ be an open cover of $Spec_B(A)$, that is 
\[Spec_B(A)=\displaystyle\bigcup_I\U(b_i)=\U(\bigvee_Ib_i).\]
Hence $1=\mu(\bigvee_Ib_i)$. Since $\mu$ is inflatory, $1=\mu(\bigvee_I\mu(b_i))$.  By Lemma \ref{lemaV}, $1=\bigvee_Ib_i$. Since $B$ is compact, there exists a finite subset $F\subseteq I$ such that $\bigvee_F b_i=1$. This implies that $Spec_B(A)=\bigcup_F\U(b_i)$.
Observe that all the instances are reversible.
\end{proof}

Using Proposition \ref{bcompspeccomp} we want to determine when $LgSpec(M)$ and $Spec(M)$ are compact spaces for a given module $M$, for we need the following lemma.

\begin{lem}\label{compatomic2}
Let $M$ be projective in $\sm$. If $\Lambda^{fi}(M)$ is compact, then $\Lambda^{fi}(M)$ is coatomic.
\end{lem}

\begin{proof}
Let $\Gamma=\{N\in\Lambda^{fi}(M)\mid N\neq M\}$ and $\mathcal{C}=\{N_i\}_I$ be a chain in $\Gamma$. If $M=\bigcup\mathcal{C}=\sum_I N_i$, then $M=\sum_{i\in F} N_i$ for some $F\subseteq I$ finite because $\Lambda^{fi}(M)$ is compact. Since $\mathcal{C}$ is a chain, $M=N_j$ for some $j\in F$ but $N_j\neq M$, a contradiction. Thus $\bigcup\mathcal{C}$ is in $\Gamma$. By Zorn's Lemma,  $\Lambda^{fi}(M)$ has maximal elements.
\end{proof}

\begin{lem}\label{maxfiinspec}
Let $M$ be  projective in $\sm$. If $\mathcal{M}$ is a maximal element in $\Lambda^{fi}(M)$, then $\mathcal{M}\in Spec(M)$.
\end{lem}

\begin{proof}
Let $\mathcal{M}$ be a maximal element in $B$. Consider $N,L\in B$ such that $N_ML\leq\mathcal{M}$. If $L\nsubseteq \mathcal{M}$ then $M=\mathcal{M}+N$. Hence, using \cite[Lemma 2.1]{maustructure} 

\vspace{4pt}\centerline{\vspace{4pt}$ N=N_MM=N_M(\mathcal{M}+N)=(N_M\mathcal{M})+(N_ML)\leq \mathcal{M}.$}

Thus, $\mathcal{M}\in Spec(M)$.
\end{proof}

\begin{cor}
Let $M$ be projective in $\sm$. If $\Lambda^{fi}(M)$ is a compact lattice then $LgSpec(M)$ and $Spec(M)$ are compact. In particular, this is satisfied when $M$ is finitely generated. 
\end{cor}

\begin{proof}
If we set $A=\Lambda(M)$ and $B=\Lambda^{fi}(M)$ then, $LgSpec(M)$ is compact by Remark \ref{casicoat} and Proposition \ref{bcompspeccomp}. On the other hand, if we set $A=B=\Lambda^{fi}(M)$, by Lemma \ref{compatomic2} $\Lambda^{fi}(M)$ is coatomic and by Lemma \ref{maxfiinspec} each coatom of $\Lambda^{fi}(M)$ is in $Spec(M)$. Hence $\Lambda^{fi}(M)$ satisfies the p-condition. It follos from Proposition \ref{bcompspeccomp} that $Spec(M)$ is compact.
\end{proof}

The following example shows a module $M$ which is not finitely generated but $\Lambda^{fi}(M)$ is compact.

\begin{ej}
Consider the abelian group $M=\mathbb{Z}_3\oplus\mathbb{Z}_2^{(X)}$ with $X$ an infinite set. It is clear that $M$ is not finitely generated. Note that 

\centerline{ $ \Lambda^{fi}(M)=\{0, \mathbb{Z}_3\oplus 0,0\oplus\mathbb{Z}_2^{(X)}, M\}.$}

\noindent Hence $Spec(M)$ and $LgSpec(M)$ are compact.
\end{ej}

Actually, as a consequence of \cite[Proposition 4.6]{HMPL2018}, it follows that $\Lambda^{fi}(M)$ is a compact idiom if and only if there exists $N\in \Lambda(M)$ finitely generated such that $\widehat{N}=M,$ where $\widehat{ N}$ denotes the least fully invariant submodule of $M$ containing $N.$

\section{Normal idioms}\label{sec3}
In this section, we study the interaction between the property of being normal in the setting of idioms  and how certain topological spaces associated with them turn out to be normal in the topological sense. 
In particular, we highlight the results obtained in Proposition \ref{pronormal} and Corollary \ref{muglobal}. Applying these results to modules, we get conditions to $Spec(M)$ and $Max(M)$  to be normal, in terms of the frames $SP(M)$ and $SPm(M),$ respectively, see Corollaries  \ref{normaleq} and \ref{NormaleqMax}

\begin{dfn}
Let $A$ be a multiplicative idiom. We say that $A$ is \emph{normal} if for every $a,b\in A$ with $a\vee b=1,$ there exist $a',b'\in A$ such that $a\vee b'=1=a'\vee b$ and $a'b'=0.$ 
\end{dfn}

\begin{lem}\label{lema1}
Let $A$ be a coatomic multiplicative idiom and $\mu\colon  A \to A$ a multiplicative nucleus which fixes every coatom. If $x,y\in A$ satisfies that $\mu(x\vee y)=1$ then $x\vee y=1$.
\end{lem}
\begin{proof}
Let $x,y\in A$ be such that $\mu(x\vee y)=1.$ If $x\vee y\lneq 1,$ then there exists a coatom $\alpha\in A$ such that $x\vee y\leq\alpha\lneq 1.$ Since $\mu$ is monotone, we obtain that $\mu(x\vee y)\leq\mu(\alpha)\lneq \mu(1)=1.$ By hypothesis, $\mu(\alpha)=\alpha$. Then,   $1=\mu(x\vee y)\leq\mu(\alpha)=\alpha\lneq 1,$ which is a contradiction.
\end{proof}

\begin{lem}\label{lema2}
Let $A$ be a coatomic multiplicative idiom and $\mu\colon  A \to A$ a multiplicative nucleus which fixes every coatom. If $A$ is normal, then $A_{\mu}$ is normal.
\end{lem}
\begin{proof}
Let $a,b\in A_{\mu}$ such that $\mu(a\vee b)=1.$ By Lemma \ref{lema1}, $a\vee b=1.$ Since $A$ is normal, it follows that there are $a',b'\in A$ satisfying that $a\vee b'=1=a'\vee b$ and $a'b'=0.$ 

Thus,

\centerline{$1=\mu(a\vee b')=\mu( \mu(a)\vee \mu(b'))=\mu(a\vee \mu(b')) ,$} 

\noindent and

\centerline{$1=\mu(a'\vee b)=\mu( \mu(a')\vee \mu(b))=\mu(\mu(a')\vee b).$}

\noindent Also, the fact that $\mu$ is a multiplicative nucleus implies that  $\mu(a')\mu(b')\leq \mu(a')\wedge\mu(b')=\mu(a'b')=\mu(0).$ Furthermore, $\mu(a')\mu(b')\leq\mu(0)$.
\end{proof}

The next Lemma gives a partial converse of Lemma \ref{lema2}.

\begin{lem}\label{mu}
Let $A$ be a coatomic multiplicative idiom and $\mu\colon  A \to A$ a multiplicative nucleus which fixes every coatom. If $A_{\mu}$ is normal and $\mu(0)=0$ then $A$ is normal.
\end{lem}
\begin{proof}
Consider any $a,b \in A$ such that $a\vee b=1$, then $\mu(a\vee b)=1$ and thus $\mu(a)=1$ and $\mu(b)=1$, that is, $a\vee_{\mu} b=1$ (this supremum is in $A_{\mu}$) then by hypothesis there exists $a', b'\in A_{\mu}$ with $a'b'=\mu(0)=0$ and $a\vee_{\mu} b'=1=a'\vee_{\mu} b$. Therefore we only need to prove that this $a',b'$  do the job in $A$, to this end if $a\vee b'\neq 1$ there exists by hypothesis a coatom $m\in A$ such that $a\vee b'< m$ then under $\mu$ we have $1=\mu(a\vee b')\leq\mu(m)=m$ which is a contradiction.
\end{proof}

\begin{obs}\label{mufix}
Let $M$ be projective in $\sm$, $A=\Lambda(M)$ and $B=\Lambda^{fi}(M)$. Consider the adjunction given in Remark \ref{adjunction}. Then $\mu$ fixes coatoms. For, let $\mathcal{M}$ be a maximal element in $B$. Then $\mathcal{M}\in Spec(M)\subseteq LgSpec(M)$ by Lemma \ref{maxfiinspec}. By Proposition \ref{41}.(a), $\mu(\mathcal{M})=\mathcal{M}$.
\end{obs}

\begin{obs}\label{adjunctionsb}
Let $B$ be a subquasi-quantale satisfying $(\star)$ of a quasi-quantale $A$. We can consider a little more general situation than that of Remark \ref{adjunction}. Given a subspace $S$ of $Spec_B(A)$, we have the hull-kernel adjunction 

\centerline{ \xymatrix@=20mm{B\ \ \ar@/^/[r]^{m} & \mathcal{O}(S)\ \ \ar@/^/[l]^{m_*}} }
 
\noindent with $m(b)=\U(b)\cap S$. Then $\tau:= m_*\circ m\colon B\to B$ is a multiplicative nucleus as in the case of $\mu$.
\end{obs}

Recall that a topological space $S$ is \emph{normal} if given two closed subsets $K$ and $L$ such that $K\cap L=\emptyset$ then there exist open subsets $U$ and $V$ with the property $K\subseteq U$, $L\subseteq V$ and $U\cap V=\emptyset$.

\begin{prop}\label{pronormal}
Let $B$ be a subquasi-quantale satisfying $(\star)$ of a quasi-quantale $A$ and let $S$ be a subspace of $Spec_B(A)$. Let $\tau$ be the multiplicative nucleus given by Remark \ref{adjunctionsb}. Then, the following conditions are equivalent.
\begin{enumerate}[label=\emph{(\alph*)}]
\item $S$ is a normal topological space.
\item $B_{\tau}$ is a normal lattice.
\end{enumerate}
\end{prop}
\begin{proof}
(a) $\Rightarrow $(b) Let $n,l\in B_{\tau}$ such that $\tau(n\vee l)=1$. Then $m(n)\cup m(l)=S$. Since $S$ is a normal space, there exist $m(k_1),m(k_2)$ open subsets such that $m(k_1)\cap m(k_2)=\emptyset$, $(S\setminus m(n))\subseteq m(k_1)$ and $(S\setminus m(l))\subseteq m(k_2)$  where $k_1,k_2\in B$. Hence $\tau(k_1)\tau(k_2)=\tau(k_1k_2)=\tau(0)$.

Note that $S=m(n)\cup m(k_1)$ and $S=m(l)\cup m(k_2)$. This implies that $\tau(n\vee k_1)=1$ and $\tau(l\vee k_2)=1$. 	Then,

\centerline{$ 1=\tau(n\vee k_1)\leq\tau(n\vee \tau(k_1)). $}

\noindent Hence $1=\tau(n\vee\tau(k_1)).$ Similarly, $1=\tau(l\vee\tau(k_2)).$ Therefore, $B_{\tau}$ is normal.

(b) $\Rightarrow$ (a)  Let $S\setminus m(n)$ and $S\setminus m(l)$ two closed sets such that $(S\setminus m(n))\cap (S\setminus m(l))=\emptyset,$ with $n,l\in B.$ Thus $m(n)\cup m(l)=S$, this implies that $\tau(\tau(n)\vee\tau(l))=1$


Since $\tau(n),\,\tau(l)\in B_{\tau}$ and $B_{\tau}$ is a normal lattice, there exist $k_1,k_2\in B_{\tau}$ such that $\tau(\tau(n)\vee k_1)=1,$ $\tau(\tau(l)\vee k_2)=1$ and $k_1k_2=\tau(0).$

Then, 

\centerline{$  (S\setminus m(n))\cap(S\setminus m(k_1))=(S\setminus m(\tau(n)))\cap(S\setminus m(k_1))=\emptyset $}

\noindent and

\centerline{\vspace{4pt} $ (S\setminus m(l))\cap(S\setminus m(k_2))=(S\setminus m(\tau(l)))\cap (S\setminus m(k_2))=\emptyset.$}

\noindent From these facts, $(S\setminus m(n))\subseteq m(k_1)$ and $(S\setminus m(l))\subseteq m(k_2).$ We have that $k_1k_2=\tau(0)$, so $\U(k_1)\cap \U(k_2)=\emptyset$. Therefore, $S$ is a normal space.  
\end{proof}  

\begin{cor}\label{muglobal}
Let $A$ be a quasi-quantale satisfying $(\star)$. Let $\mu$ be the multiplicative nucleus given by the adjoint situation on Remark \ref{adjunction}. Then, the following conditions are equivalent
\begin{enumerate}[label=\emph{(\alph*)}]
\item $Spec(A)$ is a normal space.
\item $A_\mu$ is a normal lattice
\end{enumerate} 
\end{cor}

\begin{proof}
Take $B=A$ and $S=Spec(A)$ in Proposition \ref{pronormal}.
\end{proof}

Next we give some applications to modules and rings. Recall that a proper fully invariant sbmodule $N$ of a module $M$ is said to be \emph{semiprime} if given $L\in\Lambda^{fi}(M)$ such that $L_ML\leq N$ then $L\leq N$ \cite{raggisemiprime}. In \cite[Proposition 1.11]{maugoldie} is proved that $N\in\Lambda^{fi}(M)$ is semiprime if and only if $N$ is an intersection of prime submodules, that is, an intersection of elements of $Spec(M)$. Set
\[SP(M)=\{M\}\cup\{\text{semiprime submodules}\}\subseteq\Lambda^{fi}(M).\]

In \cite[Proposition 4.27]{medina2015generalization} it is proved that $SP(M)$ is a spatial frame. 

\begin{cor}\label{normaleq}
Let $M$ be projective in $\sm$. The following conditions are equivalent:
\begin{enumerate}[label=\emph{(\alph*)}]
\item $Spec(M)$ is a normal space.
\item The frame $SP(M)$ is normal.
\end{enumerate}
\end{cor}

\begin{proof}
By Proposition \ref{41}, $\Lambda^{fi}(M)_\mu$ is the set of all submodules which are intersection of prime submodules of $M$.
\end{proof}

\begin{cor}
The following conditions are equivalent for a ring $R$:
\begin{enumerate}[label=\emph{(\alph*)}]
\item $Spec(R)$ is a normal space.
\item The frame $SP(R)$ is normal.
\end{enumerate}
\end{cor}

Recall that a proper fully invariant submodule $N$ of a module $M$ is called \emph{primitive} if $N=\Ann_M(S)$ for some simple module $S$ in $\sm$. A submodule of $M$ is called \emph{semiprimitive} if it is an intersection of primitive submodules \cite{medina2017attaching}. Set 
\[SPm(M)=\{M\}\cup\{\text{semiprimitive submodules}\}\subseteq\Lambda^{fi}(M).\]
As we said before, if $M$ is projective in $\sm$ then $\mx(M)\subseteq LgSpec(M)$. By Remark \ref{adjunctionsb}, we have a multiplicative nucleus $\tau:\Lambda^{fi}(M)\to \Lambda^{fi}(M)$. By \cite[Theorem 3.13]{medina2017attaching}, $\Lambda^{fi}(M)_\tau=SPm(M)$. Therefore, we have the following corollaries.

\begin{cor}\label{NormaleqMax}
Let $M$ be projective in $\sm$. The following conditions are equivalent:
\begin{enumerate}[label=\emph{(\alph*)}]
\item $\mx(M)$ is a normal space. 
\item The frame $SPm(M)$ is normal.
\end{enumerate}
\end{cor}

\begin{cor}\label{NormalledMaxring}
The following conditions are equivalent for a ring $R$:
\begin{enumerate}[label=\emph{(\alph*)}]
\item $\mx(R)$ is a normal space.
\item The frame $SPm(R)$ is normal.
\end{enumerate}
\end{cor}

\section{strongly harmonic modules}\label{sec4}

Throughout this section, we will be interested in to study the theory of strongly harmonic modules over associative rings with unity.

Denote the set of all coatoms in $\Lambda^{fi}(M)$ by $\mx^{fi}(M)$. Note that $\mx^{fi}(M)$ is a subspace of $Spec(M)$.

\begin{dfn}\label{strongly}
A module $M$ is \emph{strongly harmonic} if for every distinct elements $N,L\in \mx^{fi}(M)$ there exist $N',L'\in\Lambda^{fi}(M)$ such that $L'\nleq L,\,N'\nleq N$ and $L'_MN'=0.$ 
\end{dfn}

\begin{prop}\label{STRONGLY}
	Let $M$ be a module such that $-_M-$ is an associative product. Then, the following conditions are equivalent.
	\begin{enumerate}[label=\emph{(\alph*)}]
		\item $M$ is strongly harmonic.
		\item For every distinct elements $N,L\in \mx^{fi}(M)$ there exist $N',L'\in\Lambda(M)$ such that $L'\nleq L,\,N'\nleq N$ and $L'_MN'=0.$
		\item For every distinct elements $N,L\in \mx^{fi}(M)$ there exist $a,b\in M$ such that $a\notin N,$ $b\notin L,$ $a\in \Ann_M(Rb)$.  
	\end{enumerate}
\end{prop}
\begin{proof}  
(a) $\Rightarrow$ (b) It is clear.
 
(b) $\Rightarrow$ (c) Let $N,L\in \mx^{fi}(M).$ By (b), there exist $N',L'\in\Lambda(M)$ such that $L'\nleq L,\,N'\nleq N$ and $L'_MN'=0$. In particular,  
there exist $a\in L'\backslash L$ and $b\in N'\backslash N.$ Consequently, $Ra\leq L'$ and $Rb\leq N'.$ Then $Ra_M Rb\leq L'_MN'=0.$
Hence, $a\in Ra\leq \Ann_M(Rb).$ 

(c) $\Rightarrow$ (a) Let $N,L\in \mx^{fi}(M).$ There exist $a,b\in M$ such that $a\notin N,$ $b\notin L,$ $a\in \Ann_M(Rb)$ and   $b\in \Ann_M(Ra).$ Set $N':=Ra_MM$ and $L':=Rb_MM$. Note that $N',L'\in \Lambda^{fi}(M),$  $N'\nleq N$ and $L'\nleq L.$ 
Now, using the associativity of the product $-_M-,$ we have that 
$L'_MN'={(Ra_MM)}_M(Rb_MM)={({(Ra_MM)}_MRb)}_MM.$	

Inasmuch as  $a\in \Ann_M(Rb)\in\Lambda^{fi}(M),$ it follows that $Ra\leq \Ann_M(Rb)$ and so ${Ra}_MM\leq {\Ann_M(Rb)}_MM=\Ann_M(Rb).$ Thus,

\vspace{4pt}\centerline{$ L'_MN'={({(Ra_MM)}_MRb)}_MM\leq ({\Ann_M(Rb)}_MRb)_MM=0.$}
\end{proof}

\begin{lem}[Lemma 17, \cite{raggiprime}]\label{lemmapp}
Let $M_1,M_2$ be modules and let $f:M_1\to M_2$ be an epimorphism with $K_1=Ker f$.
\begin{enumerate}[label=\emph{(\alph*)}]
\item If $K_1$ is a fully invariant in $M_1$ and $N_2$ is a fully invariant submodule of $M_2,$ then $f^{-1}(N_2)$ is a fully invariant submodule of $M_1$.
\item If $M_1$ is quasi-projective and $N_1$ is a fully invariant submodule of $M_1,$ then $f(N_1)$ is a fully invariant submodule of $M_2$.
\end{enumerate}
\end{lem}

Recall that a nonzero module $M$ is called \emph{FI-simple} if $\Lambda^{fi}(M)=\{0,M\}$. The following lemma will be useful, and it is a direct consequence of Lemma \ref{lemmapp}.

 \begin{lem}\label{maxFI}
Let $M$ be a quasi-projective module and $N\in\Lambda^{fi}(M)$. Then, $N\in\mx^{fi}(M)$ if and only if $M/N$ is FI-simple.
\end{lem}

\begin{prop}\label{summandfa}
Let $M$ be a quasi-projective strongly harmonic module and let $N$ be a fully invariant submodule of $M$. Then $M/N$ is a strongly harmonic module.
\end{prop}

\begin{proof}
Let $\mathcal{M}/N,\mathcal{N}/N\in\mx^{fi}(M/N)$ be distinct. It follows from Lemma \ref{lemmapp} and Lemma \ref{maxFI} that $\mathcal{M},\mathcal{N}\in\mx^{fi}(M)$. Since $M$ is strongly harmonic, there exist $A,B\in\Lambda^{fi}(M)$ such that $A\nleq\mathcal{M}$, $B\nleq\mathcal{N}$ and $A_MB=0$. Since $A\nleq\mathcal{M}$, $(A+N)/N\nleq\mathcal{M}/N$. Analogously, $(B+N)/N\nleq\mathcal{N}/N$. We claim that the product $\left(\frac{A+N}{N}\right)_{M/N}\left(\frac{B+N}{N}\right)=0$. Let $f\colon M/N\to (B+N)/N$ be any homomorphism. Since $M$ is quasi-projective, there exists $\overline{f}:M\to B$ such that $\pi|_{B}\overline{f}=f\pi$ where $\pi:M\to M/N$ is the canonical projection. Note that $\overline{f}(A)=0$ because $A_MB=0$. Hence,
$ f\left(\displaystyle\frac{A+N}{N}\right)=f\pi(A)=\pi|_B\overline{f}(A)=0.$
This proves the claim. Thus, $M/N$ is a strongly harmonic module.
\end{proof}

\begin{cor}
Let $R$ be a strongly harmonic ring and let $I$ be an ideal of $R$. Then the ring $R/I$ is strongly harmonic.
\end{cor}

\begin{cor}
Let $R$ be a strongly harmonic ring and $e\in R$ be a central idempotent. Then $Re$ is a strongly harmonic module.
\end{cor}

\begin{obs}\label{isosumafi}
Given a module $M$ and $N\in\Lambda^{fi}(M),$ there exists a preradical $\alpha$ in $R$-Mod such that $N=\alpha(M)$, see \cite{raggilattice}. Then, for every index set $I$ there exists a lattice isomorphism $\Theta:\Lambda^{fi}(M)\to\Lambda^{fi}(M^{(I)})$ given by $\Theta(N)=N^{(I)}$. Note that this isomorphism restricts to a bijection $\Theta:\mx^{fi}(M)\to\mx^{fi}(M^{(I)})$ provided $M$ is quasi-projective by Lemma \ref{maxFI}. 
\end{obs}

\begin{prop}\label{dirsumfa}
Let $M$ be a quasi-projective strongly harmonic module. Then $M^{(I)}$ is a strongly harmonic module for every index set $I$.
\end{prop}

\begin{proof}
Let $N\neq L\in\mx^{fi}(M^{(I)})$. Then there exist $A,B\in\mx^{fi}(M)$ such that $N=A^{(I)}$ and $L=B^{(I)}$. By hypothesis there exist $K_1,K_2\in\Lambda^{fi}(M)$ such that $K_1\nleq A$, $K_2\nleq B$ and ${K_1}_MK_2=0$. Hence $K_1^{(I)}\nleq N$ and $K_2^{(I)}\nleq L$. If $f:M^{(I)}\to K_2$ is any morphism, then $f((m_i)_I)=\sum_If(\eta_i(m_i))$ where $\eta_i:M\to M^{(I)}$ are the canonical inclusions. Since ${K_1}_MK_2=0$ then $f\eta_i(K_1)=0$ for all $i\in I$. Hence $f(K_1^{(I)})=0$. This implies that 

\[{K_1^{(I)}}_{M^{(I)}}K_2^{(I)}=\left({K_1^{(I)}}_{M^{(I)}}K_2\right)^{(I)}=0.\]

\noindent Thus $M^{(I)}$ is strongly harmonic.
\end{proof}

\begin{cor}
Let $R$ be a strongly harmonic ring. Then every right (left) free $R$-module is strongly harmonic.
\end{cor}

\begin{prop}\label{shdirsum}
Let $M$ be a quasi-projective module. Suppose $M=\bigoplus_I M_i$ is a direct sum with $M_i\in\Lambda^{fi}(M)$. Then $M$ is a strongly harmonic module if and only if $M_i$ is strongly harmonic.
\end{prop}

\begin{proof}
Suppose $M_i$ is strongly harmonic for every $i\in I$. Let $N, L\in\mx^{fi}(M)$ distinct. There exist preradicals $\alpha$ and $\beta$ in $R$-Mod such that 

{\vspace{5pt}}
\centerline{$N=\alpha(M)=\alpha\left(\bigoplus_IM_i\right)=\bigoplus_I\alpha(M_i)$}

and 

\centerline{$L=\beta(M)=\beta\left(\bigoplus_IM_i\right)=\bigoplus_I\beta(M_i).$}
{\vspace{5pt}}

By Lemma \ref{maxFI} and \cite[Lemma 17]{raggiprime},  there exist $i,k\in I$ such that $\alpha(M_i)\neq M_i$ and $\alpha(M_j)=M_j$ for all $j\neq i$, and $\beta(M_k)\neq M_k$ and $\beta(M_j)=M_j$ for all $j\neq k$. Thus,

\vspace{4pt}\centerline{\vspace{4pt}$N=\bigoplus_{j\neq i}M_j\oplus \alpha(M_i), \mbox{ and } \, L=\bigoplus_{j\neq k}M_j\oplus\beta(M_k).$}

Note that $M_i\nleq N$ and $M_k\nleq L$. By \cite[Proposition 1.8]{PepeGab}, $\Ann_M(M_i)=\bigoplus_{j\neq i}M_j$. If $i\neq k$ then $M_k\leq \Ann_M(M_i),$ and so ${M_k}_MM_i=0$. On the other hand, suppose $i=k$. Since $N\neq L$, $\alpha(M_i)\neq\beta(M_i)$. Note that $\alpha(M_i),\beta(M_i)\in\mx^{fi}(M_i)$. By hypothesis, there exist $A,B\in\Lambda^{fi}(M_i)$ such that $A\nleq\alpha(M_i)$, $B\nleq\beta(M_i),$ and $A_{M_i}B=0$. Consider $\eta_i(A)$ and $\eta_i(B),$ the images of $A$ and $B$ under the canonical inclusion $\eta_i:M_i\to M,$ respectively. Then $\eta_i(A),\eta_i(B)\in\Lambda^{fi}(M)$. Let $f:M\to \eta_i(B)$ be any homomorphism. Hence, $\pi_if\eta_i:M_i\to B,$ where $\pi_i:M\to M_i$ is the canonical projection. Since $A_{M_i}B=0$, $\pi_if(\eta_i(A))=0$. We have that $f(\eta_i(A))\leq\eta_i(B)\leq M_i$, so $\pi_jf(\eta_i(A))=0$ for all $j\neq i$. This implies that $f(\eta_i(A))=0$, that is $\eta_i(A)_M\eta_i(B)=0$. Thus $M$ is strongly harmonic. 

The converse follows from Proposition \ref{summandfa}.
\end{proof}

It is easy to see that, in general, the direct sum of two strongly harmonic modules is not strongly harmonic. For instance, 

\begin{ej}
Let $R=\left(\begin{smallmatrix}
\mathbb{Z}_2 & \mathbb{Z}_2 \\
0 & \mathbb{Z}_2
\end{smallmatrix}\right)$. Consider $e_1=\left(\begin{smallmatrix}
1 & 0 \\
0 & 0
\end{smallmatrix}\right)$ and $e_2=\left(\begin{smallmatrix}
0 & 0 \\
0 & 1
\end{smallmatrix}\right)$. Then $Re_1=\left(\begin{smallmatrix}
\mathbb{Z}_2 & 0 \\
0 & 0
\end{smallmatrix}\right)$ and $Re_2=\left(\begin{smallmatrix}
0 & \mathbb{Z}_2 \\
0 & \mathbb{Z}_2
\end{smallmatrix}\right)$. Note that $Re_1$ is simple and $Re_2$ has three submodules $\{0, \left(\begin{smallmatrix}
0 & \mathbb{Z}_2 \\
0 & 0
\end{smallmatrix}\right), \left(\begin{smallmatrix}
0 & \mathbb{Z}_2 \\
0 & \mathbb{Z}_2
\end{smallmatrix}\right)\}$. Hence $R=Re_1\oplus Re_2$ is a direct sum of strongly harmonic modules. Also, $R$ has two maximal fully invariant submodules: $Re_2=\left(\begin{smallmatrix}
\mathbb{Z}_2 & \mathbb{Z}_2 \\
0 & 0
\end{smallmatrix}\right),$ and $\mathcal{M}=\left(\begin{smallmatrix}
0 & \mathbb{Z}_2 \\
0 & \mathbb{Z}_2
\end{smallmatrix}\right)$. The unique nonzero proper ideal not contained in $Re_2$ is $\mathcal{M}$ and the unique nonzero proper ideal not contained in $\mathcal{M}$ is $Re_2$. Note that $Re_2\mathcal{M}\neq 0$. Thus, $R$ is not strongly harmonic.
\end{ej}

By Lemma \ref{maxfiinspec}, $\mx^{fi}(M)$ is contained in $Spec(M)$. Given $K\in\Lambda^{fi}(M)$, the open subset relative to $\mx^{fi}(M)$ is denoted by $m(K)=\U(K)\cap\mx^{fi}(M)$. 

\begin{prop}\label{stmaxhausdorf}
Let $M$ be projective in $\sm$. If $M$ is strongly harmonic, then $\mx^{fi}(M)$ is a Hausdorff subspace of $Spec(M).$ If in addition, $0=\bigcap\mx^{fi}(M)$ then converse holds. 
\end{prop}

\begin{proof} Consider the topological subspace $\mx^{fi}(M)$ and   $N_1,N_2\in \mx^{fi}(M).$   Since $M$ is strongly harmonic, there exist $L_1,\,L_2\in \Lambda^{fi}(M)$ such that $L_1\nleq N_1,$   $L_2\nleq N_2$ and ${L_1}_M{L_2}=0.$ 
Then,  $N_1\in m(L_1)$ and $N_2\in m(L_2).$  Also, $m(L_1)\cap  m(L_2)
=m({L_1}_M{L_2}).$ Inasmuch as ${L_1}_M{L_2}=0,$ we can conclude that $m({L_1}_M{L_2})=\emptyset.$ Therefore, $\mx^{fi}(M)$ is Hausdorff.

Reciprocally, assume that $0=\bigcap\mx^{fi}(M)$. Let $N\neq L\in\mx^{fi}(M)$. Since $\mx^{fi}(M)$ is Hausdorff, there exist disjoint open sets $m(K_1)$ and $m(K_2)$ of $\mx^{fi}(M)$ containing $N$ and $L$ respectively. This implies that ${K_1}_MK_2\subseteq\bigcap\mx^{fi}(M)=0$. Note that $K_1\nleq N$ and $K_2\nleq L$. Thus, $M$ is strongly harmonic.
\end{proof}

\begin{lem}\label{kohmodules}
Let $M$ be projective in $\sigma[M]$ and strongly harmonic. If $\mathcal{F}\subseteq \mx^{fi}(M)$ is a compact subset and $N\in \mx^{fi}(M)$ is such that $N\notin \mathcal{F},$ then there exists $L,K\in\Lambda^{fi}(M)$ with $N\in m(K)$ and $F\subseteq m(L)$ such that $L_MK=0$. Moreover, if $\mx^{fi}(M)$ is compact , then $\mx^{fi}(M)$ is a normal space.
\end{lem}

\begin{proof}
Let $\F{:=}\{N_{\alpha}\}_{\alpha\in \Gamma}$ a compact subset of $\mx^{fi}(M)$ and $N{\in }\mx^{fi}(M)$ such that $N\notin \mathcal{F}$.  Since $M$ is strongly harmonic, for each $\alpha\in\Gamma$ there exist $L_{\alpha},\,K_{\alpha}\in\Lambda^{fi}(M)$ such that $L_{\alpha}\nleq N_{\alpha},$ $K_{\alpha}\nleq N,$ and ${L_{\alpha}}_{M}K_{\alpha}=0.$ 
Hence, $\{m(L_{\alpha})\}_{\alpha\in\Gamma}$ is an open cover for $\F.$
Since $\F$ is compact, there exist $\alpha_{1},\ldots,\alpha_n\in\Gamma$ sucht that 

\vspace{4pt}\centerline{\vspace{4pt}$\displaystyle\F\subseteq \bigcup_{i=1}^n\{m(L_{\alpha_i})\}=m(\sum_{i=1}^nL_{\alpha_i}).$}

\noindent Also, we have that  $N\in \bigcap_{i=1}^n m(K_{\alpha_i})=m({K_{\alpha_1}}_M\cdots_M {K_{\alpha_n}})$ and  ${L_{\alpha_i}}_{M}K_{\alpha_i}=0,$ for each $i=1,\ldots,n.$ 

Now, using the facts that $-_{M}-$ is associative and $M$ is a right multiplicative identity on $\Lambda^{fi}(M)$,  

\vspace{4pt}\centerline{\vspace{4pt}$ {L_{\alpha_2}}_M({K_{\alpha_1}}_M{K_{\alpha_2}})\leq {L_{\alpha_2}}_M({M_M}K_{\alpha_2})=({L_{\alpha_2}}_MM)_MK_{\alpha_2}={L_{\alpha_2}}_M{K_{\alpha_2}}=0.
$}

Thus, ${L_{\alpha_2}}_M({K_{\alpha_1}}_M{K_{\alpha_2}})=0.$ Similarly, it can be proved that 
\[{L_{\alpha_3}}_M({K_{\alpha_1}}_M{K_{\alpha_2}}_M{K_{\alpha_3}})\leq {L_{\alpha_3}}_MK_{\alpha_3}=0.\]
In fact, it is satisfied that ${L_{\alpha_i}}_M({K_{\alpha_1}}_M{K_{\alpha_2}}\cdots _M{K_{\alpha_i}})=0,$ for every $i\in \{ 1,\ldots, n\}.$ Consequently, ${(\sum_{i=1}^n{L_{\alpha_i}})}_M ( {K_{\alpha_1}}_M\cdots_M {K_{\alpha_n}})=0.$ Therefore,  $L:=\sum_{i=1}^n{L_{\alpha_i}}$ and $K:= {K_{\alpha_1}}_M\cdots_M {K_{\alpha_n}}$ are the required modules. 

Since  $M$ is strongly harmonic, $\mx^{fi}(M)$ is Hausdorff. If in addition, we have $\mx^{fi}(M)$ is compact, then every closed set is compact. It gets that $\mx^{fi}(M)$ is a compact Hausdorff regular space, by the above argument. It is well known, from general topology theory, that those conditions imply that the underlying space is normal. 
\end{proof}

\begin{lem}\label{compatomic}
Let $M$ be projective in $\sm$. $\Lambda^{fi}(M)$ is compact if and only if $\mx^{fi}(M)$ is compact and $\Lambda^{fi}(M)$ is coatomic.
\end{lem}

\begin{proof}
$\Rightarrow$ By Lemma \ref{compatomic2}, $\Lambda^{fi}(M)$ is coatomic. Now,  suppose that  $\mx^{fi}(M)=\bigcup_I\U(N_i)=\U(\sum_I N_i)$. This implies that $\sum_I N_i\nleq \mathcal{M}$ for all $\mathcal{M}\in\mx^{fi}(M)$. Since $\Lambda^{fi}(M)$ is coatomic, $\sum_I N_i=M$. Hence $M=\sum_{i\in F} N_i$ for some $F\subseteq I$ finite by hypothesis. Thus $\mx^{fi}(M)=\bigcup_F\U(N_i)$, that is, $\mx^{fi}(M)$ is compact.

$\Leftarrow$ Let $M=\sum_I N_i$ with $N_i\in\Lambda^{fi}(M)$. Then $\mx^{fi}(M)=\bigcup_I\U(N_i)$. Since $\mx^{fi}(M)$ is compact, $\mx^{fi}(M)=\bigcup_{i\in F}\U(N_i)$ for some $F\subseteq I$ finite. This implies that $\sum_{F} N_i\nleq \mathcal{M}$ for all $\mathcal{M}\in\mx^{fi}(M)$. Since $\Lambda^{fi}(M)$ is coatomic, $M=\sum_F N_i$. That is, $\Lambda^{fi}(M)$ is compact.
\end{proof}

\begin{prop}\label{norstr}
Let $M$ be projective in $\sigma[M]$. Consider the following conditions.
\begin{enumerate}[label=\emph{(\alph*)}]
\item $\Lambda^{fi}(M)$ is normal.
\item $M$ is strongly harmonic. 
 
\end{enumerate}
Then \emph{(a)}$\Rightarrow$\emph{(b)} holds. If in addition, $\Lambda^{fi}(M)$  is compact, the two conditions are equivalent.
\end{prop}
\begin{proof}
(a)$\Rightarrow$(b) Let $N_1,N_2\in \mx^{fi}(M),$ with $N_1\neq N_2.$ Then $N_1+N_2=M.$ 
By the normality on $\Lambda^{fi}(M),$ it follows  
 that there exist $L_1,L_2\in \Lambda^{fi}(M)$ such that    
 $N_1+L_1=M$ and $N_2+L_2=M,$ and ${L_1}_M{L_2}=0.$ 
Hence,  $M$ is strongly harmonic.

Now suppose that $\Lambda^{fi}(M)$ is a compact space.\\
(b)$\Rightarrow$(a) Let $L_1,L_2\in \Lambda^{fi}(M)$ such that $L_1+L_2=M.$ Hence $\U(L_1)\cup\U(L_2)=Spec(M)$, it follows that $\V(L_1)\cap \V(L_2)=\emptyset.$ So, $\V(L_1)\cap \mx^{fi}(M)$ and $\V(L_2)\cap \mx^{fi}(M)$ are disjoint closed sets in $\mx^{fi}(M).$ It follows from Lemma \ref{compatomic} that $\mx^{fi}(M)$ is compact. Then by Lemma \ref{kohmodules} and Proposition \ref{stmaxhausdorf},  there exist $K_1,K_2\in \Lambda^{fi}(M)$ such that $\V(L_1)\cap \mx^{fi}(M)\subseteq m(K_1)$, $\V(L_2)\cap \mx^{fi}(M)\subseteq m(K_2)$ and    ${K_1}_M{K_2}=0.$

Notice that if  $L_1+K_1<M,$ there exists $N\in \mx^{fi}(M)$ such that $L_1+K_1\leq N$ by Lemma \ref{compatomic}. But this implies that $N\in\V(L_1)\cap \V(K_1)$ which is a contradiction. Therefore $L_1+K_1=M$. Analogously, $L_2+K_2=M$. Thus $\Lambda^{fi}(M)$ is normal.
\end{proof}

\begin{cor}\label{ringlamstr}
The following conditions are equivalent for a ring $R$:
\begin{enumerate}[label=\emph{(\alph*)}]
\item $\Lambda^{fi}(R)$ is normal.
\item $R$ is a strongly harmonic ring.
\end{enumerate}
\end{cor}

 Recall that, setting $A=B=\Lambda^{fi}(M)$ in Remark \ref{adjunction} we have a multiplicative nucleus $\mu:\Lambda^{fi}(M)\to \Lambda^{fi}(M)$. We can resume the applications of Section 3 to modules and rings and these section's results in the following theorem and corollary.

\begin{thm}\label{theoremsth} 
Let $M$ be projective in $\sigma[M]$ such that $\Lambda^{fi}(M)$ is compact. Consider the following conditions:
\begin{enumerate}[label=\emph{(\alph*)}]
\item $M$ is a strongly harmonic module.
\item $\Lambda^{fi}(M)$ is normal.
\item $\Lambda^{fi}(M)_{\mu}$ is a normal lattice. 
\item $Spec(M)$ is a normal space.
\end{enumerate}
Then the implications \emph{(a)}$\Rightarrow$\emph{(b)}$\Rightarrow$\emph{(c)}$\Leftrightarrow$\emph{(d)} hold. If in addition $0=\bigcap\mx^{fi}(M)$, then the four conditions are equivalent.
\end{thm}

\begin{proof}
(a) $\Leftrightarrow$ (b) It follows from Proposition \ref{norstr}.

(b) $\Rightarrow$ (c) It follows from Lemma \ref{lema2}.

(c) $\Leftrightarrow$ (d) It follows from Corollary \ref{muglobal}.

Now suppose that $0=\bigcap\mx^{fi}(M)$. 

(d)$\Rightarrow$ (a) Since $Spec(M)$ is normal, $\mx^{fi}(M)$ is Hausdorff. Hence $M$ is strongly harmonic by Proposition \ref{stmaxhausdorf}.
\end{proof}

\begin{cor}
Let $R$ be a ring such that the intersection of all maximal ideals is zero. The following conditions are equivalent:
\begin{enumerate}[label=\emph{(\alph*)}]
\item $R$ is a strongly harmonic ring.
\item $Spec(R)$ is a normal space.
\item The lattice of ideals of $R$ is a normal lattice.
\item The frame of semiprime ideals is a normal lattice. 
\end{enumerate}
\end{cor}

Given an $R$-module $M$, in \cite[Section 5]{medina2017attaching} was defined the spatial frame $\Psi(M),$ as follows:
\[\Psi(M)=\{N\in\Lambda^{fi}(M)\mid\forall n\in N,[N+\Ann_M(Rn)]=M\}.\]
If $M$ is self-progenerator in $\sm$, the frame $\Psi(M)$ is characterized as the fixed points of an operator called $Ler:\Lambda^{fi}(M)\to \Lambda^{fi}(M)$ \cite[Proposition 5.11]{medina2017attaching}. This operator is defined as 
\[Ler(N)=\{m\in M\mid N+\Ann_M(Rm)=M\},\]
for $N\in\Lambda^{fi}(M)$.
Properties of this operator are given in \cite[Propositions 5.8--5.10]{medina2017attaching}. The operator $Ler$ will be crucial to give a connection between the frames $\mathcal{O}(\mx^{fi}(M))$ and $\Psi(M)$ for a strongly harmonic module $M$.

\begin{obs}\label{lersuma}
For a module $M$ projective in $\sm$, the operator $Ler$ can be described as
\[Ler(N)=\sum\{K\in\Lambda^{fi}(M)\mid N+Ann_M(K)=M\},\]
for any $N\in\Lambda^{fi}(M)$.
\end{obs}

Given a module $M$ projective in $\sm$ and $N\in\Lambda^{fi}(M),$ there exists the greatest (fully invariant) submodule $\Ann_M^r(N)$ of $M$ such that $N_M\Ann_M^r(N)=0,$ see \cite[Definition 1.14]{mauacc}.

\begin{lem}\label{ler1}
Let $M$ be projective in $\sm,$ and suppose that $\Lambda^{fi}(M)$ is coatomic. Let $N\in\Lambda^{fi}(M)$ and $\mathcal{M}\in \mx^{fi}(M)$. If $M$ is strongly harmonic then the following statements hold, 
\begin{enumerate}[label=\emph{(\alph*)}]
\item If $Ler(\mathcal{M})\leq N\neq M$ then $N\leq \mathcal{M}$.
\item If $M=N+\Ann_M^r(L)$ then $L\leq Ler(N)$.
\end{enumerate}
\end{lem}

\begin{proof}
(a) Suppose that $Ler(\mathcal{M})\leq N$ and $N\neq M$. Then there exists $\mathcal{N}\in\mx^{fi}(M)$ such that $N\leq \mathcal{N}$. If $\mathcal{M}\neq\mathcal{N}$ then there exist $N',L'\in\Lambda^{fi}(M)$ such that $N'\nleq \mathcal{M}$ and $L'\nleq\mathcal{N}$ such that ${N'}_ML'=0$. Since $N'\leq\Ann_M(L')$ then $\Ann_M(L')\nleq\mathcal{M}$. Hence $M=\mathcal{M}+\Ann_M(L')$. Therefore, 

\centerline{$L'\leq Ler(\mathcal{M})\leq N\leq \mathcal{N},$}

\noindent and this is a contradiction. Thus $\mathcal{N}=\mathcal{M}$.

(b)  Let $L\leq M$ such that $M=N+\Ann_M^r(L)$. Suppose $N+\Ann_M(L)\neq M$. Then, there exists $\mathcal{M}\in\mx^{fi}(M)$ such that $N+\Ann_M(L)\leq\mathcal{M}$. Note that $L\nleq Ler(\mathcal{M})$. Set $K=Ler(\mathcal{M})+\Ann_M^r(L)$. By (a), $K=M$ or $K\leq\mathcal{M}$. Suppose $K=M$, then

\centerline{$L\leq L_MM= L_MLer(\mathcal{M})+L_M\Ann_M^r(L)= L_MLer(\mathcal{M})\leq Ler(\mathcal{M}),$}

\noindent getting a contradiction. Now, if $K\leq\mathcal{M}$ then $\Ann_M^r(L)\leq\mathcal{M}$. Also, we have that $N\leq\mathcal{M}$. Hence $M=N+\Ann_M^r(L)\leq\mathcal{M},$ a contradiction. Thus, $M=N+\Ann_M(L),$ that is, $L\leq Ler(N)$.
\end{proof}

\begin{thm}\label{dreamtheo}
Let $M$ be a self-progenerator in $\sm$. Assume  $\Lambda^{fi}(M)$ is compact. The following conditions are equivalent:
\begin{enumerate}[label=\emph{(\alph*)}]
\item $M$ is strongly harmonic.

\item $\Lambda^{fi}(M)$ is a normal idiom.

\item For each $N\in\Lambda^{fi}(M)$ and $\mathcal{M}\in \mx^{fi}(M)$

\centerline{$ Ler(N)\leq\mathcal{M}\Leftrightarrow N\leq\mathcal{M}.$}

\item $Ler$ is $\sum$-preserving (equivalently $Ler$ has right adjoint)

\item For each $N, L\in\Lambda^{fi}(M)$ 

\centerline{$ N+L=M\Rightarrow Ler(N)+Ler(L)=M. $}

\end{enumerate}
\end{thm}

\begin{proof}
(a) $\Leftrightarrow$(b) It follows from Lemma \ref{compatomic} and Proposition \ref{norstr}.

(a) $\Rightarrow$ (c) Suppose $N\nleq\mathcal{M}$. Then, $M=N+Ler(\mathcal{M})$ by Lemma \ref{ler1}(a). So, 
$M=N+\sum\{K\in\Lambda^{fi}(M)\mid \mathcal{M}+\Ann_M(K)=M\}$ by Remark \ref{lersuma}.
Since $\Lambda^{fi}(M)$ is compact,

\centerline{ $M=N+\displaystyle \sum_{i=1}^n\{K_i\mid \mathcal{M}+\Ann_M(K_i)=M\}.$}
\noindent Therefore, 
\begin{equation*}
\begin{split}
M &=M_MM_M\cdots_MM \\
& =(\mathcal{M}+\Ann_M(K_1))_M(\mathcal{M}+\Ann_M(K_2))_M\cdots_M(\mathcal{M}+\Ann_M(K_n)) \\
& \leq \mathcal{M}+\bigcap_{i=1}^n \Ann_M(K_i).
\end{split}
\end{equation*}

On the other hand, each $K_i\leq\Ann^r_M(\Ann_M(K_i))$. Hence,
\begin{equation*}
\begin{split}
M &=N+\sum_{i=1}^n\{K_i\mid \mathcal{M}+\Ann_M(K_i)=M\}\\
& \leq N+\sum_{i=1}^n\{\Ann^r_M(\Ann_M(K_i))\mid\mathcal{M}+\Ann_M(K_i)=M\}.
\end{split}
\end{equation*}
Note that $\sum_{i=1}^n\Ann^r_M(\Ann_M(K_i))\leq\Ann_M^r(\Ann_M(\sum_{i=i}^{n}K_i))$. So,

\centerline{$ M=N+\Ann_M^r(\Ann_M(\displaystyle\sum_{i=i}^{n}K_i)).$}

\noindent By Lemma \ref{ler1}(b), $\Ann_M(\sum_{i=i}^{n}K_i)\leq Ler(N)\leq\mathcal{M}$. Since $\bigcap_{i=1}^n \Ann_M(K_i)=\Ann_M(\sum_{i=i}^{n}K_i)$, 

\centerline{$M=\mathcal{M}+\displaystyle \bigcap_{i=1}^n \Ann_M(K_i)=\mathcal{M}+\Ann_M(\sum_{i=i}^{n}K_i)\leq\mathcal{M},$}

\noindent a clear contradiction. Thus $N\leq\mathcal{M}$. The converse is clear.\\

(c) $\Rightarrow$ (d) Let $\{N_i\}_I\subseteq\Lambda^{fi}(M)$. It follows from \cite[Lemma 5.9]{medina2017attaching} that,

\centerline{$\displaystyle\sum_I Ler(N_i)\leq Ler\left(\sum_I N_i\right).$}

\noindent Let $a\in Ler(\sum_I N_i)$, then $M=\sum_I N_i+\Ann_M(Ra)$. Suppose that $M\neq \sum_I Ler(N_i)+\Ann_M(Ra)$. Then, there exists $\mathcal{M}\in\mx^{fi}(M)$ such that 

\centerline{$\displaystyle \sum_I Ler(N_i)+\Ann_M(Ra)\leq\mathcal{M}.$}

\noindent In particular, $Ler(N_i)\leq\mathcal{M}$ for all $i\in I$. By hypothesis, $N_i\leq\mathcal{M}$ for all $i\in I.$ Thus, 

\centerline{$ M=\displaystyle \sum_I Ler(N_i)+\Ann_M(Ra)\leq\mathcal{M}$}

\noindent a contradiction. Thus $M=\sum_I Ler(N_i)+\Ann_M(Ra)$, that is, 

\centerline{$\displaystyle a\in Ler\left(\sum_ILer(N_i)\right)\leq\sum_I Ler(N_i).$}

\noindent Therefore,

\centerline{$\displaystyle\sum_I Ler(N_i)=Ler\left(\sum_I N_i\right).$}

(d) $\Rightarrow$ (e) Suppose that $M=N+L$ with $N,L\in\Lambda^{fi}(M)$. Then 

\vspace{4pt}\centerline{\vspace{4pt} $ M=Ler(M)=Ler(N+L)=Ler(N)+Ler(L).$}

(e) $\Rightarrow$ (a) Let $\mathcal{M},\mathcal{N}\in\mx^{fi}(M)$ be distinct. Then $M=\mathcal{M}+\mathcal{N}$. If $Ler(\mathcal{M})\leq\mathcal{N}$, then
$ M=\mathcal{M}+\mathcal{N}=Ler(\mathcal{M})+Ler(\mathcal{N})\leq \mathcal{N},$
 a contradiction. Hence, there exists $a\in Ler(\mathcal{M})$ such that $a\notin \mathcal{N}$. We have that $M=\mathcal{M}+\Ann_M(Ra)$, hence $\Ann_M(Ra)\nleq\mathcal{M}$. Therefore, $Ra\nleq\mathcal{N}$, $\Ann_M(Ra)\nleq\mathcal{M}$ and $\Ann_M(Ra)_MRa=0$. Consequently, $M$ is strongly harmonic.
\end{proof}

\begin{lem}\label{ler2}
Let $M$ be a self-progenerator in $\sm$. Assume $M$ is strongly harmonic such that $\Lambda^{fi}(M)$ is compact. Then $Ler$ is idempotent.
\end{lem}

\begin{proof}
Let $m\in Ler(N)$, that is, $M=N+\Ann_M(Rm)$. Suppose that $Ler(N)+\Ann_M(Rm)\neq M$. Then, there exists $\mathcal{M}\in\mx^{fi}(M)$ such that
 $Ler(N)+\Ann_M(Rm)\leq\mathcal{M}$. 
 So, $Ler(N)\leq\mathcal{M}$. By Theorem \ref{dreamtheo}(3) $N\leq\mathcal{M}$. Thus, $M=N+\Ann_M(Rm)\leq\mathcal{M},$ a contradiction. Consequently, $Ler(N)+{\Ann_M}(Rm)=M,$ and $m\in Ler(Ler(N))$.
\end{proof}

Now we can give a connection for a strongly harmonic module$M$ between $\Psi(M)$ and $\mathcal{O}(\mx^{fi}(M))$. In the next Proposition we prove that the point space of $\Psi(M)$ is homeomorphic to the space $\mx^{fi}(M)$.

\begin{prop}\label{homeo}
Let $M$ be a self-progenerator in $\sm$. Assume that $\Lambda^{fi}(M)$ is compact. Then $\pt(\Psi(M))$ is homeomorphic to $\mx^{fi}(M)$.
\end{prop}

\begin{proof}
Let $\mathcal{M}\in\mx^{fi}(M)$. We claim that $Ler(\mathcal{M})\in\pt(\Psi(M))$. Let $N,L\in\Psi(M)$ such that $N\cap L\leq Ler(\mathcal{M})$. Since $N,L\in\Lambda^{fi}(M)$, $N_ML\leq N\cap L$. This implies $N_ML\leq\mathcal{M}$. Therefore, $N\leq \mathcal{M}$ or $L\leq\mathcal{M}$. By \cite[Proposition 5.11]{medina2017attaching}, $N=Ler(N)\leq Ler(\mathcal{M})$ or $L=Ler(L)\leq Ler(\mathcal{M}),$ proving the claim. 
Define $\Theta:\mx^{fi}(M)\to \pt(\Psi(M))$ as $\Theta(\mathcal{M})=Ler(\mathcal{M})$. Suppose $Ler(\mathcal{M})=Ler(\mathcal{N})$ for $\mathcal{M},\mathcal{N}\in\mx^{fi}(M)$. Hence $Ler(\mathcal{M})\leq\mathcal{N}$. By Lemma \ref{ler1} $\mathcal{N}\leq\mathcal{M}$. Thus $\mathcal{M}=\mathcal{N}$, that is, $\Theta$ is injective.
Let $U(N)$ be an open set of $\pt(\Psi(M))$. Then

\vspace{4pt}\centerline{$\mathcal{M}\in\Theta^{-1}(U(N))\Leftrightarrow Ler(\mathcal{M})\in U(N)\Leftrightarrow N\nsubseteq Ler(\mathcal{M})$}

\centerline{\vspace{4pt}$\Leftrightarrow N\nsubseteq \mathcal{M}\text{ because }N\text{ is a fixed point of }Ler.$}

\noindent Thus $\Theta^{-1}(U(N))=m(N)$, that is, $\Theta$ is continuous.

Let $N\in\pt(\Psi(M))$. We claim that $N$ is contained in a unique element of $\mx^{fi}(M)$. Suppose $N\leq\mathcal{M}$ and $N\leq\mathcal{N}$ with $\mathcal{M},\mathcal{N}\in\mx^{fi}(M)$. If $\mathcal{M}\neq\mathcal{N}$ there exist $A,B\in\Lambda^{fi}(M)$ such that $A\nleq\mathcal{M}$, $B\nleq\mathcal{N}$ and $A_MB=0$. Hence $Ler(A)_MLer(B)=0$. Since $Ler$ is idempotent by Lemma \ref{ler2}, $Ler(A),Ler(B)\in\Psi(M).$ Then,  by \cite[Proposition 5.4]{medina2017attaching},

\vspace{4pt}\centerline{\vspace{4pt}$Ler(A)\cap Ler(B)=Ler(A)_MLer(B)\leq N.$}

\noindent Thus, $Ler(A)\leq N\leq\mathcal{M}$ or $Ler(B)\leq N\leq\mathcal{N}$. By Theorem \ref{dreamtheo}(3), $A\leq\mathcal{M}$ or $B\leq\mathcal{N}$ which is a contradiction. Therefore, $\mathcal{M}=\mathcal{N}$ proving the claim.

Let $N\in\pt(\Psi(M))$ and let $\mathcal{M}\in\mx^{fi}(M)$ such that $N\leq\mathcal{M}$. Suppose that $N$ is contained properly in $Ler(\mathcal{M})$. Then there is $a\in Ler(M)$ with $a\notin N$. Since $a\notin N=Ler(N)$, $M\neq N+\Ann_M(Ra)$. Hence, there exists $\mathcal{N}\in\mx^{fi}(M)$ such that $N+\Ann_M(Ra)\leq\mathcal{N}$. By the claim proved above, $\mathcal{N}=\mathcal{M}$. Hence $\Ann_M(Ra)\leq\mathcal{M}$. On the other hand, since $a\in Ler(\mathcal{M})$, $M=\mathcal{M}+\Ann_M(Ra)$. Thus $M=\mathcal{M}+\Ann_M(Ra)\leq\mathcal{M}$ which is a contradiction. Therefore, $N=Ler(\mathcal{M})$. This proves that the function $\Theta$ is surjective. Moreover, $\Theta$ is an open map. For, let $m(K)$ be an open set in $\mx^{fi}(M)$. Then

\vspace{4pt}\centerline{\vspace{4pt} $\Theta(m(K))=\{Ler(\mathcal{M})\mid \mathcal{M}\in m(K)\}=\{Ler(\mathcal{M})\mid K\nsubseteq \mathcal{M}\}$}

\centerline{\vspace{4pt} $=\{Ler(\mathcal{M})\mid K\nsubseteq Ler(\mathcal{M})\}=\{N\in\pt(\Psi(M))\mid K\nsubseteq N\}.$}

Thus $\Theta\colon\mx^{fi}(M)\to \pt(\Psi(M))$ is a homeomorphism.
\end{proof}

\begin{cor}
	If $R$ is a strongly harmonic ring, then $\pt(\Psi(R))$ is homeomorphic to $\mx^{fi}(R)$.
\end{cor}

In \cite[Theorem 5.20]{medina2017attaching} was studied the regularity of the frame $\Psi(M)$ in the sense of \cite{johnstone1986stone} and \cite{simmons1989compact}. Here, we can give other conditions to get the regularity of that frame. Note that by Lemma \ref{ler2} $Ler$ is idempotent. Hence, for $N\leq M$, $Ler(N)$ is the largest submodule of $N$ in the frame $\Psi(M)$.

\begin{thm}\label{dreamcon}
Let $M$ be a self-progenerator in $\sm$. Assume $M$ is strongly harmonic such that $\Lambda^{fi}(M)$ is compact. Then $\Psi(M)$ is regular, that is, $\Psi(M)=(\Lambda^{fi}(M))^{reg}$.
\end{thm}

\begin{proof}
Let $r:\Psi(M)\to \Psi(M)$ given by

\centerline{$r(N)=\sum\{K\in\Psi(M)\mid N+K^r=M\},$}

\noindent where $K^r=\sum\{L\in\Psi(M)\mid L_MK=0\}$. Note that $K^r\leq\Ann_M(K)$. Note that $Ler(\Ann_M(K))\in\Psi(M)$ and $Ler(\Ann_M(K))\leq\Ann_M(K)$ then $Ler(\Ann_M(K))\leq K^r$. On the other hand, $Ler(\Ann_M(K))$ is the largest submodule of $\Ann_M(K)$ in $\Psi(M)$, hence $K^r=Ler(\Ann_M(K))$. This implies that
{\begin{equation*}
\begin{split}
r(N) & =\sum\{K\in\Psi(M)\mid N+K^r=M\}\\
& =\sum\{K\in\Psi(M)\mid N+Ler(\Ann_M(K))=M\}\\
& =\sum\{K\in\Psi(M)\mid Ler(N)+Ler(\Ann_M(K))=M\}\\
& =\sum\{K\in\Psi(M)\mid Ler(N+\Ann_M(K))=M\}\\
& =\sum\{K\in\Psi(M)\mid N+\Ann_M(K)=M\}
\end{split}
\end{equation*}}
We have that $\Ann_M(\_)$ is order-reversing and $Ler$ commutes with sums (Theorem \ref{dreamtheo}), 
\vspace{-5pt}
\begin{equation*}
\begin{split}
Ler(N)& =Ler(Ler(N))\\
& =Ler\left(\sum\{B\in\Lambda^{fi}(M)\mid N+\Ann_M(B)=M\}\right)\\
& =\sum\{Ler(B)\mid N+\Ann_M(B)=M\}\\
& \leq\sum\{Ler(B)\mid N+\Ann_M(Ler(B))=M\}\\
& \leq r(N)
\end{split}
\end{equation*}
Since $Ler(N)=N$, $N\leq r(N)$. Thus $\Psi(M)=(\Lambda^{fi}(M))^{reg}$.
\end{proof}
 
Let $\EuScript{KHT}\mathrm{op}$ be the category of compact Hausdorff spaces and continuous functions. It is well known that this category is dually equivalent to the category $\EuScript{KRF}\mathrm{rm}$ of compact regular frames and frames morphisms (see \cite{banassatone1} or \cite{johnstone1986stone} and \cite{picado2011frames}).

\begin{cor}\label{corhomeo}
Let $M$ be a self-progenerator in $\sm$. Assume $M$ is strongly harmonic such that $\Lambda^{fi}(M)$ is compact. Then $\Psi(M)\cong\mathcal{O}(\mx^{fi}(M))$
\end{cor}

\begin{proof}
By Proposition \ref{dreamcon} that $\Psi(M)$ is a compact regular frame with associated space $\mx^{fi}(M)$ (Proposition \ref{homeo}). It follows from Proposition \ref{stmaxhausdorf} that this space is compact Hausdorff, and so 
\[\Psi(M)\cong\mathcal{O}(\mx^{fi}(M)).\]\end{proof}

\begin{cor}
	If $R$ is a strongly harmonic ring then $\Psi(R)\cong\mathcal{O}(\mx^{fi}(R))$.
\end{cor}

\begin{lem}\label{roler}
	Let $M$ be projective in $\sm$ and let $\rho\colon M\to N$ be any homomorphism. Then $\rho Ler(L)\leq Ler\rho(L)$ for any $L\in\Lambda^{fi}(M)$.
\end{lem}

\begin{proof}
	Let $K\leq M$. If $H_MK=0$ for some $H\leq M$, then $H_M\rho(K)=0$ by \cite[Lemma 5.9]{beachy2002m}. Therefore, $g\rho(H)=0$ for all $g\colon N\to \rho(K),$ and consequently,  $\rho(H)_N\rho(K)=0$. This implies that $\rho(H)\leq\Ann_N(\rho(K))$. It follows $\rho(\Ann_M(K))\leq \Ann_N(\rho(K))$. Hence, if $m\in Ler(L)$, that is, $M=L+\Ann_M(Rm)$ then 	
	
\centerline{$N=\rho(M)=\rho(L)+\rho(\Ann_M(Rm))\leq \rho(L)+\Ann_N(R\rho(m)).$}
	
\noindent	Thus, $\rho(m)\in Ler(\rho(L)).$ Therefore, $\rho Ler(L)\leq Ler\rho(L)$.
\end{proof}

\begin{lem}\label{idmorf}
	Let $M$ be quasi-projective and $\rho\colon M\to N$ be any epimorphism. If $N$ is a strongly harmonic module, self-progenerator in $\sigma[N]$ and $\Lambda^{fi}(N)$ is compact, then $Ler\rho\colon \Lambda^{fi}(M)\to \Psi(N)$ is an idiom morphism.
\end{lem}

\begin{proof}
	It follows from Theorem \ref{dreamtheo} that $Ler\rho$ commutes with sums. Now, let $L$ and $K$ in $\Lambda^{fi}(M)$. By Lemma \ref{lemmapp}, $\rho(L),\rho(K)\in\Lambda^{fi}(N)$ and it is clear that
	
\vspace{4pt}\centerline{\vspace{4pt} $Ler\rho(L\cap K)\leq Ler\rho(L)\cap Ler\rho(K).$}
	
	\noindent Let $n\in Ler\rho(L)\cap Ler\rho(K)$. Then $N=\rho(L)+\Ann_{N}(Rn)$ and $N=\rho(K)+\Ann_N(Rn)$. We claim that $\rho(L)_N\rho(K)\leq \rho(L_MK)$. Let $g:N\to f(K)$ be any homomorphism. Since $M$ is quasi-projective in $\sm$, there exists $h_g\colon M\to K$ such that $g\rho=\rho h_g$. Therefore,
	\begin{equation*}
	\begin{split}
	\rho(L)_N\rho(K) & =\sum\left\{g\rho(L)\mid g\colon N\to \rho(K)\right\} \\
	 & =\sum\left\{\rho h_g(L)\mid h_g\colon M\to K\right\} \\
	 & =\rho\left(\sum\left\{h_g(L)\mid h_g\colon M\to K\right\}\right) \\
	 & \leq \rho\left(L_MK\right). 	
	\end{split}
	\end{equation*}
	That proves our claim. On the other hand,
	\begin{equation*}
	\begin{split}
	N=N_NN & = \left(\rho(L)+\Ann_N(Rn)\right)_M\left(\rho(K)+\Ann_N(Rn)\right) \\
	 & \leq \rho(L)_N\rho(K) + \Ann_N(Rn) \\
	 & \leq \rho(L_MK) + \Ann_N(Rn) \\
	 & \leq \rho(L\cap K) +\Ann_N(Rn)
	\end{split}
	\end{equation*}
	Hence $n\in Ler(\rho(L\cap K))$. Thus, $Ler\rho$ commutes with finite intersections.
\end{proof}

\begin{prop}\label{framemor}
Let $M$ be self-progenerator in $\sm$ and $\rho\colon M\to N$ be any epimorphism. If $N$ is a strongly harmonic module, self-progenerator in $\sigma[N]$ and $\Lambda^{fi}(N)$ is compact, then $\rho:\Psi(M)\to \Psi(N)$ is a frame morphism.
\end{prop}

\begin{proof}
	By Lemma \ref{idmorf} we just have to prove that $\rho=Ler\rho$. Let $L\in\Psi(M)$. Then $Ler(L)=L$ by \cite[Proposition 5.11]{medina2017attaching}. It follows using Lemma \ref{roler} that
 $\rho(L)=\rho(Ler(L))\leq Ler(\rho(L))\leq \rho(L).$ Thus, $\rho(L)=Ler(\rho(L))$.
\end{proof}

\begin{cor}\label{corh}
Let $M$ be a self-progenerator in $\sm$. Assume $M$ is strongly harmonic such that $\Lambda^{fi}(M)$ is compact. Then, there exists an frame morphism $\Psi(R)\to \Psi(M)$.
\end{cor}

\begin{proof}
We know that there exists a free module $R^{(X)}$ and an epimorphism $\rho:R^{(X)}\to M$. Hence $\rho$ defines a frame morphism $\Psi(R^{(X)})\to \Psi(M)$ by Proposition \ref{framemor}. Note that, by Remark \ref{isosumafi} $\Psi(R)\cong\Psi(R^{(X)})$.
\end{proof}

\begin{obs}\label{selfprocciente}
If $M$ is self-progenerator in $\sm$ and $N\in\Lambda^{fi}(M)$, then $M/N$ is a self-progenerator in $\sigma[M/N]$.
\end{obs}

\begin{prop}\label{dual}
Let $M$ be a self-progenerator in $\sm$. Assume $M$ is strongly harmonic such that $\Lambda^{fi}(M)$ is compact. Then, the  assignment \[\Lambda^{fi}(M)\rightarrow \EuScript{KRF}\mathrm{rm}\] given by $N\mapsto \Psi(M/N)$ determines a functor.
\end{prop}

\begin{proof}
By Remark \ref{selfprocciente}, $M/N$ is a self-progenerator in $\sigma[M/N].$ Also, by Proposition \ref{summandfa}, $M/N$ is a strongly harmonic module,  and then satisfies the hypothesis of the Theorem \ref{dreamcon}. Given $N\leq L$, there is an epimorphism $M/N\to M/L$. Hence by Proposition \ref{framemor}, there is a frame morphism $\Psi(M/N)\to \Psi(M/L)$.
\end{proof}

\begin{prop}\label{dual0}
Let $R$ be a ring, and let $\EuScript{SH}^{fi}$ denote   the subcategory of $R\Mod$ whose objects are all strongly harmonic modules $M$ satisfying that they are self-progenerator in their $\sm$ and $\Lambda^{fi}(M)$ is a compact idiom, and whose morphisms are epimorphism $\rho\colon M\to N$. Then, the $\Psi(\_)$ construction provides a covariant functor

\vspace{4pt}\centerline{$ \Psi(\_)\colon\EuScript{SH}^{fi}\rightarrow\EuScript{KRF}\mathrm{rm} $}
\end{prop}

\begin{proof}
The result follows immediately from \ref{framemor}.
\end{proof}

\section{Gelfand modules}\label{sec5}

On this section, we introduce the concept of Gelfand modules, in an attempt to give a modular version of the existing concept for rings, and we obtain some characterizations of these. As in the case of rings, we note that each Gelfand module turns out to be also strongly harmonic.

\begin{obs}\label{propertyeta}
Let $N\leq M$. In \cite{medina2015generalization} was considered the preradical $\eta^M_N(K)=\bigcap\{f^{-1}(N)\mid f\in\Hom_R(K,M)\}$. It was proved that if $P\in LgSpec(M),$ then $\eta^M_P(M)\in Spec(M)$ \cite[Proposition 4.9 and Proposition 4.10]{medina2015generalization}. In particular, we have that $\eta^M_{\mathcal{M}}(M)\in Spec(M)$ for any maximal submodule $\mathcal{M}$ of $M$.
\end{obs}

\begin{prop}
Let $N$ be a submodule of a module $M$. Then, $\eta^M_N(M)$ is the greatest fully invariant submodule of $M$ which is contained in $N.$
\end{prop}

\begin{proof} Let $m\in \eta^M_N(M).$ Then $m=Id(m)\in N.$ Thus, $\eta^M_N(M)\leq N.$

Now, let $m\in \eta^M_N(M)$ and $g\in \End_R(M).$
 Let $f\in \End_R(M)$, $f(g(m))=fg(m)\in N.$ Then $g(m)\in \eta^M_N(M).$ Therefore, $\eta^M_N(M)\in \Lambda^{fi}(M).$
 
 Finally, let $K$ be a fully invariant submodule of $M$ such that $K\leq N.$ Then, for any $f\in \End_R(M)$ we have that $f(K)\leq K\leq N.$ Thus, $K\leq \eta^M_N(M).$
\end{proof}

\begin{lem}\label{Mf}
Let $M$ be a module and $\mathcal{M}< M$ be a maximal submodule. Then $f^{-1}(M)$ is a maximal submodule of $M$ for all $f\in\End_R(M)$ such that $f(M)\nsubseteq \mathcal{M}$.
\end{lem}
\begin{proof}
Let $f\in\End_R(M)$ such that $f(M)\nsubseteq \mathcal{M}$. Then there exists $x\in M$ such that $f(x)\notin \mathcal{M}$. Therefore, $f$ defines an isomorphism

\vspace{4pt}\centerline{\vspace{4pt} $ M/f^{-1}(\mathcal{M})\to M/\mathcal{M}=R(f(x)+\mathcal{M}).$}

\noindent Thus, $f^{-1}(\mathcal{M}) $ is maximal.
\end{proof}

\begin{dfn}\label{defGelfand}
A module $M$ is \emph{Gelfand} if for every distinct elements $N,L\in \mx(M)$ there exist $N',L'\in\Lambda^{fi}(M)$ such that $L'\nleq L,\,N'\nleq N$ and $L'_MN'=0.$ 
\end{dfn}

\begin{prop}\label{Gelfand}
	Let $M$ be a module such that $-_M-$ is an associative product. Then, the following conditions are equivalent.
	\begin{enumerate}[label=\emph{(\alph*)}]
		\item $M$ is a Gelfand module.
		\item For every distinct elements $N,L\in \mx(M)$ there exist $N',L'\in\Lambda(M)$ such that $L'\nleq L,\,N'\nleq N$ and $L'_MN'=0.$
		\item For every distinct elements $N,L\in \mx(M),$ there exist $a\notin L$ and $b\notin N$ such that  for each $f:M\to Rb,\,$ $f(a)=0$ holds, that is $a\in Ann_M(Rb).$
	\end{enumerate}
\end{prop}

\begin{proof}
It is similar to the proof of Proposition \ref{STRONGLY}.
\end{proof}

\begin{lem}\label{Gispm}
Let $M$ be a Gelfand module and $P\leq M$ be a prime submodule. If there exist $L,N\in\mx(M)$ such that $P\leq N$ and $P\leq L$ then $N=L$. 

\end{lem}
\begin{proof} Let $P$ be a prime module of $M$ and let $L,N\in\mx(M)$ satisfying that $L\neq N$ and $P\leq L\cap N.$ Since $M$ is Gelfand, there exist $L',N'$ such that $L'\neq L,$ $N'\neq N$ and ${L'}_M N'=0.$
Then, $0={L'}_MN'\leq P.$ Because $P$ is prime, it follows that $L'\leq P\leq L$ or $N'\leq P\leq N,$ which is a contradiction.
Thus, $L=N.$
\end{proof}

\begin{prop}\label{MaxGelfand}
Let $M$ be a  Gelfand module. Then $M$ is a quasi-duo module (i.e $\mx(M)\subseteq \Lambda^{fi}(M))$.
\end{prop}
\begin{proof}
Let $\mathcal{M}\in\mx(M).$ Then $\eta^M_{\mathcal{M}}(M)\leq \mathcal{M}$ is a prime submodule of $M.$ We claim that 
$\eta^M_{\mathcal{M}}(M)=\mathcal{M}.$ Let $f\in\End_R(M)$. If $f(M)\leq\mathcal{M}$, then $f^{-1}(\mathcal{M})=M$. This implies that 

\vspace{4pt}\centerline{\vspace{4pt} $\eta^M_\mathcal{M}(M)=\bigcap\{f^{-1}(\mathcal{M})\mid f\in\End_R(M)\text{ and }f^{-1}(M)\nsubseteq\mathcal{M}\}.$}

Now, let $f\in\End_R(M)$ such that $f(M)\nsubseteq\mathcal{M}$. By Lemma \ref{Mf}, $f^{-1}(\mathcal{M})$ is a maximal submodule of $M$. Since $\eta^M_{\mathcal{M}}(M)\leq \mathcal{M}$ and $\eta^M_{\mathcal{M}}(M)\leq f^{-1}(\mathcal{M})$, $\mathcal{M}=f^{-1}(\mathcal{M})$ by Lemma \ref{Gispm}. Thus, $\mathcal{M}=\eta^M_{\mathcal{M}}(M).$ Therefore, $\mathcal{M}$ is fully invariant.
\end{proof}

\begin{cor}\label{gelcoa}
Let $M$ be projective in $\sm$. If $M$ is a Gelfand module then $\Lambda^{fi}(M)$ is coatomic.
\end{cor}

\begin{proof}
Let $N\in\Lambda^{fi}(M)$. Since $M$ is projective in $\sm$, there exists $\mathcal{M}\in\mx(M)$ such that $N\leq\mathcal{M}$ by Remark \ref{casicoat}. By Proposition \ref{MaxGelfand}, $\mathcal{M}\in\Lambda^{fi}(M)$. Thus $\Lambda^{fi}(M)$ is coatomic.
\end{proof}

\begin{lem}\label{Glemma10.11H}
Let $M$ be a quasi-duo module projective in $\sm$. Then $\mx^{fi}(M)=\mx(M).$	
\end{lem}
\begin{proof}
Let $\mathcal{N}\in\mx(M).$  By hypothesis, $\mathcal{N}\in\Lambda^{fi}(M).$ We notice that $\mathcal{N}\in\Lambda^{fi}(M).$ Indeed, let $L\in\Lambda^{fi}(M)$ such that $\mathcal{N}\leq L.$ 
Since $\mathcal{N}\in \mx(M),$ we get $\mathcal{N}=L.$ Hence, $\mx(M)\subseteq \mx^{fi}(M).$

On the other hand, let $\mathcal{K}\in\mx^{fi}(M)\subseteq \Lambda^{fi}(M).$ By Remark \ref{casicoat}, there exists $\mathcal{M}\in\mx(M)$ such that $K\leq \mathcal{M}.$ By hypothesis, $\mx(M)\subseteq \Lambda^{fi}(M).$ So $\mathcal{M}\in \Lambda^{fi}(M)$ and $\mathcal{K}\in\mx^{fi}(M)$ implies that $K=\mathcal{M}.$
\end{proof}

The following result allows us to note that in order to study the Gelfand modules, we can  focus first on the study of strongly harmonic modules that satisfy the extra condition of being quasiduo.

\begin{thm}\label{SHDuoGelfand}
Let $M$ be projective in $\sm$. Then following conditions are equivalent: 
\begin{enumerate}[label=\emph{(\alph*)}]
\item $M$ is Gelfand
\item $M$ is strongly harmonic and quasi-duo.
\end{enumerate}
\end{thm}
\begin{proof}
(a) $\Rightarrow$ (b) Let $M$ be a Gelfand module. By Proposition \ref{MaxGelfand}, $M$ is quasi-duo. It follows by Lemma \ref{Glemma10.11H} that $\mx(M)=\mx^{fi}(M)$. Thus, $M$ is strongly harmonic by Proposition \ref{Gelfand}.

(b) $\Rightarrow$ (a) We have that $\mx(M)=\mx^{fi}(M)$ by Lemma \ref{Glemma10.11H}. Hence $M$ is a Gelfand module by Proposition \ref{STRONGLY}.
\end{proof}

\begin{prop}\label{summandgel}
Let $M$ be a quasi-projective Gelfand module and $N\leq M$. Then $M/N$ is a Gelfand module.
\end{prop}

\begin{proof}
Let $\mathcal{M}/N,\mathcal{N}/N\in\mx(M/N)$. It follows that $\mathcal{M},\mathcal{N}\in\mx(M)$. Since $M$ is a Gelfand module, there exist $A,B\in\Lambda^{fi}(M)$ such that $A\nleq\mathcal{M}$, $B\nleq\mathcal{N}$ and $A_MB=0$. Since $A\nleq\mathcal{M}$, $(A+N)/N\nleq\mathcal{M}/N$. Analogously, $(B+N)/N\nleq\mathcal{N}/N$. We claim that the product $\left(\frac{A+N}{N}\right)_{M/N}\left(\frac{B+N}{N}\right)=0$. Let $f:M/N\to (B+N)/N$ be any homomorphism. Since $M$ is quasi-projective, there exists $\overline{f}:M\to B$ such that $\pi|_{B}\overline{f}=f\pi$ where $\pi:M\to M/N$ is the canonical projection. Note that $\overline{f}(A)=0$ because $A_MB=0$. Hence,

\vspace{4pt}\centerline{$\displaystyle f\left(\frac{A+N}{N}\right)=f\pi(A)=\pi|_B\overline{f}(A)=0.$}

\noindent This proves the claim. Thus, $M/N$ is a Gelfand module.
\end{proof}

\begin{cor}
The following conditions are equivalent for a ring $R$:
\begin{enumerate}[label=\emph{(\alph*)}]
\item $R$ is a Gelfand ring.
\item Every cyclic $R$-module is Gelfand.
\item $Re$ is a Gelfand module for any idempotent $e\in R$.
\end{enumerate}
\end{cor}

\begin{obs}
In contrast to strongly harmonic modules (Proposition \ref{dirsumfa}),  an arbitrary coproduct of copies of a Gelfand module might not be Gelfand. In fact, direct sums of copies of a quasi-duo module is not quasi-duo in general, as the following example shows: the semisimple $\mathbb{Z}$-module $M=\mathbb{Z}_2\oplus\mathbb{Z}_3$ is Gelfand. Note that $\mathbb{Z}_2\oplus\mathbb{Z}_2\oplus\mathbb{Z}_3$ is a maximal submodule of $M\oplus M$ which is not fully invariant.
\end{obs}

Recall that from Remark \ref{adjunctionsb}, setting $A=\Lambda(M)$, $B=\Lambda^{fi}(M)$ and $S=\mx(M)$, we have a multiplicative nucleus $\tau:\Lambda^{fi}(M)\to \Lambda^{fi}(M)$. Notice that Proposition \ref{homeo}, in particular ensures that the frames $\Psi(M)$ and $\mathcal{O}(\mx(M))\cong SPm(M)$ are isomorphic for the case of Gelfand modules. We have to notice that, for the case of Gelfand rings, this was proved in \cite[Theorem 4.1]{borceux1984sheaf}. We can give a direct proof of that fact and see that $\tau$ and $Ler$ define an isomorphism between those two frames. For, we need the next Lemma.

\begin{lem}
Let $M$ be a self-progenerator in $\sm$. Assume $M$ is a Gelfand module and $\mx(M)$ is compact. Then $Ler(N)\leq L$ if and only if $N\leq\tau(L)$ for all $N,L\in\Lambda^{fi}(M)$.
\end{lem}

\begin{proof}
Let $N,L\in\Lambda^{fi}(M)$. Assume $Ler(N)\leq L$. Let $\mathcal{M}\in\mx(M)$ such that $L\leq\mathcal{M}$. Hence $Ler(N)\leq\mathcal{M}$. By Corollary \ref{gelcoa} and Lemma \ref{compatomic}, the lattice $\Lambda^{fi}(M)$ is compact. Therefore Theorem \ref{dreamtheo}(3) implies that $N\leq\mathcal{M}$. Since $\mathcal{M}$ is any maximal submodule containing $L$, $N\leq\tau(L)$.

Conversely, suppose that $N\leq\tau(L)$. Let $m\in Ler(N)$. Then $M=N+Ann_M(Rm)$. If $m\notin Ler(L)$, then $L+Ann_M(Rm)\neq M$. By Corollary \ref{gelcoa}, there exists $\mathcal{M}\in\mx(M)$ such that $L+Ann_M(Rm)\leq\mathcal{M}$. This implies that $\tau(L)\leq\mathcal{M}$ and by hypothesis $N\leq\mathcal{M}$. Thus, $M=N+Ann_M(Rm)\leq\mathcal{M}$. Contradiction. Hence $m\in Ler(L)$ and so $Ler(N)\leq L$.
\end{proof}

\begin{thm}\label{DreamcomG}
Let $M$ be a self-progenerator in $\sm$. Assume $M$ is a Gelfand module and $\mx(M)$ is compact. Then $\Psi(M)\cong SPm(M)$ as frames.
\end{thm}

\begin{proof}
We claim that $\tau Ler(L)=\tau(L)$ for all $L\in SPm(M)$. Since $Ler(L)\leq L$, $\bigcap\{\mathcal{M}\in\mx(M)\mid L\leq \mathcal{M}\}\subseteq\bigcap\{\mathcal{M}\in\mx(M)\mid Ler(L)\leq \mathcal{M}\}$. By Theorem \ref{dreamtheo}(3), any maximal submodule containing $Ler(L)$ contains $L$, hence 

\vspace{4pt}\centerline{\vspace{4pt}$ \bigcap\{\mathcal{M}\in\mx(M)\mid L\leq \mathcal{M}\}=\bigcap\{\mathcal{M}\in\mx(M)\mid Ler(L)\leq \mathcal{M}\}.$}

\noindent This implies that $\tau Ler(L)=\tau(L)$.

Now we claim that $Ler\tau(N)=Ler(N)$ for all $N\in SPm(M)$. Since $N\leq\tau(N)$ then $Ler(N)\leq Ler\tau(N)$. On the other hand, let $m\in Ler\tau(N)$. Then $M=\tau(N)+Ann_M(Rm)$. If $m\notin Ler(N)$, then $M\neq N+Ann_M(Rm)$ and so there exists $\mathcal{M}\in\mx(M)$ such that $N+\Ann_M(Rm)\leq\mathcal{M}$. Since $\tau(N)$ is contained in every maximal submodule which contains $N$, $\tau(M)\leq\mathcal{M}$. Therefore $M=\tau(N)+Ann_M(Rm)\leq\mathcal{M}$. Contradiction, proving the claim.

The above two claims imply that $Ler\tau=Id_{\Psi(M)}$ and $\tau Ler=Id_{SPm(M)}$. Since $\tau$ is a nucleus and by \cite[Proposition 5.10]{medina2017attaching}, $\tau$ and $Ler$ commutes with finite intersections.  It remains to prove that $Ler$ and $\tau$ are $\bigvee$-preserving. Recall that the supremum of a family $\{N_i\}_I$ in the frame $SPm(M)$ is given by $\tau(\sum_I N_i)$. Hence $\displaystyle Ler(\tau(\sum_I N_i))=Ler(\sum_I N_i)=\sum_I(Ler(N_i))$ by Theorem \ref{dreamtheo}(4). Thus $Ler$ is $\bigvee$-preserving. On the other hand, let $\{L_i\}_I$ be a family of elements in $\Psi(M)$. Using Theorem \ref{dreamtheo}(4) we get

\centerline{ $\displaystyle \tau\left(\sum_I L_i\right)=\tau\left(\sum_I Ler(L_i)\right)=\tau\left(\sum_I Ler\tau(L_i)\right)$}
\centerline{\vspace{4pt} $=\displaystyle \tau Ler\left(\sum_I \tau(L_i)\right)=\tau\left(\sum_I \tau(L_i)\right).$}

\noindent Thus $\tau$ is $\bigvee$-preserving. Therefore, the frames $\Psi(M)$ and $SPm(M)$ are isomorphic.
\end{proof}

\begin{cor}
Let $R$ be a Gelfand ring. Then $\Psi(R)\cong SPm(R)$ as frames.
\end{cor}

Given $S$ a topological space and $A$ a subspace of  $S,$ recall that a continuous map $\gamma\colon S\to A$ is a retraction if $\gamma \circ \iota_A=1_A, $ where $\iota_A$ denotes the canonical inclusion map of $A$ into $S.$

\begin{prop}\label{retractnormal}
Let $M$ be satisfying that $\mx(M)$ is compact.
If $\mx(M)$ is Hausdorff and  it is a retract of $Spec(M)$, then  $Spec(M)$ is normal.
\end{prop}
\begin{proof}
Let $\gamma\colon Spec(M)\to \mx (M)$ a continuous retraction. We claim $Spec(M)$ is normal.
First, notice that if $F$ is closed in $Spec(M)$, then $\gamma(F)=F\cap \mx(M).$
 So, given $F_1,\,F_2$  closed sets in $Spec(M)$ then $F_1\cap \mx(M)$ and $F_2\cap \mx(M)$ are closed in $\mx(M).$ 
 Now, recall that a space which is Hausdorff and compact turns out to be normal. Thus, there are $U_1,\,U_2$ open disjoint sets in $\mx(M)$ satisfying that $\gamma(F_1)\subseteq U_1$ and $\gamma(F_2)\subseteq U_2.$ Thus,  $F_1\subseteq \gamma^{-1}(U_1)$ and $F_2\subseteq \gamma^{-1}(U_2).$
 \end{proof}


\begin{obs}\label{etacontinua} The map $\eta\colon LgSpec(M)\to Spec(M)$ given by $\eta(Q):=\eta^M_Q(M)$ is a surjective, continuous, and closed function.
Indeed, by  \cite[Proposition 4.9 and Proposition 4.10]{medina2015generalization}, it follows that $\eta$ is well defined.
Now, take $\mathcal{V}_{Spec(M)}(N)$ a basic closed subset in $Spec(M).$ Then, \[\eta(\mathcal{V}_{Spec(M)}(N))^{-1}=\{Q\in LgSpec(M)\mid \eta^M_Q(M)\in\mathcal{V}_{Spec(M)}(N)\}.\]
By definition of $\mathcal{V}_{Spec(M)}(N)$ and using the fact that $\eta^M_Q(M)\leq Q,$ we conclude that $\eta(\mathcal{V}_{Spec(M)}(N))^{-1}=\mathcal{V}_{LgSpec(M)}(N).$ 
It is clear that $\eta$ is surjective.
Finally, notice that $\eta$ is a closed function. Let $\mathcal{V}_{LgSpec(M)}(N)$ is a basic closed  set  in $LgSpec(M).$ Since $N\in\Lambda^{fi}(M)$ by Remark \ref{propertyeta} it follows that $N\leq \eta^M_Q(M)$ for each $Q\in\mathcal{V}_{LgSpec(M)}.$
 Thus, $\eta({\mathcal{V}}_{LgSpec(M)}(N))= \mathcal{V}_{Spec(M)}(N).$ 
\end{obs}

\begin{cor} If $LgSpec(M)$ is normal, then $Spec(M)$ is normal.
\end{cor}
\begin{proof} It is a consequence of Remark \ref{etacontinua} and the fact that the continuous and closed image of a normal space is normal.\end{proof}

\begin{prop}
Let $M$ be  projective in $\sigma[M]$ such that  $Spec(M)$ is 
normal, $\Lambda^{fi}(M)$ compact, with $\mu(0)=0.$ 
then $\mx^{fi}(M)$ is Hausdorff and  it is a retract of $Spec(M).$
\end{prop}
\begin{proof}
 
 Define 
 $\gamma\colon Spec(M)\to \Max^{fi}(M),$ give by 
 
 \centerline{$\gamma(P)=\sum\{N\in\Lambda^{fi}(M)\mid N+P < M\}.$ }

\noindent Let us see that $\gamma$ is well defined. 
Let $P\in Spec(M)$ and suppose that $\gamma(P)=M.$ Since  $\Lambda^{fi}(M)$ is compact, $M=\sum_{i=1}^nN_i,$ where 
every $N_i\in\{N\in \Lambda^{fi}(M)\mid N+P < M\}.$ By induction on $n,$ it can prove that due $P$ is prime and $\Lambda^{fi}(M)$ is normal (by Lemma \ref{lema2} and Lemma \ref{mu}),  there exists $N_i$ satisfying that $N_i+P=M,$ which is a contradiction to the hypothesis on $N_i$.  This contradiction comes from the assupmtion $\gamma(P)=M.$ Thus, $\gamma(P)< M.$
Now, note $\gamma(P)\in Max^{fi}(M)$. For, suppose that $\gamma(P)< K,$ since $K\nleq \gamma(P),$ then, we get $K+P=M.$ Also, due $P$ is prime, it follows that  $P\in \{N\in \Lambda^{fi}(M)\mid N+P < M\},$ and so $P\leq \gamma(P).$ Thus, $K+\gamma(P)=K+P+\gamma(P)=M.$
And so, $K+\gamma(P)=M.$
Thus, $\gamma(P)\in \Max^{fi}(M).$ If, in particular, $P$ is maximal, then $\gamma(P)=P.$

Finally, we will see $\gamma$ is continuous.  Let $\mathcal{U}(N)$ a basic open set of $Spec(M),$ and consider $\emph{m}(N):=\mathcal{U}(N)\cap Max^{fi}(M)$ a basic open set in $Max^{fi}(M).$ First, let $P\in \gamma^{-1}(\emph{m}(N))=\{L\in Spec(M)\mid \gamma(L)\in \emph{m}(N)\}=
\{L\in Spec(M)\mid \gamma(L)+N=M\}.$ So, $P+N=M.$ Since $\Lambda^{fi}(M)$ is normal, there exists $K_1,\,K_2$ such that $N+K_1=M=P+K_2$ and ${K_1}_M{K_2}=0.$ Because of $P$ is a prime submodule, we also get $K_2\not\subseteq P,$ $K_1\subseteq P.$  Thus, $P\in \U(K_2),$ where $\U(K_2)$ denotes a basic open set of $Spec(M).$
Hence, $\gamma^{-1}(m(N)\subseteq \U(K_2).$
Now, let $Q\in \U(K_2).$ So, $K_2\not\subseteq Q.$ Since $Q$ is prime and ${K_1}_M{K_2}=0,$ we get $K_1\subseteq Q.$ Also, by a previous analysis on $\gamma,$ we also know that $Q\subseteq \gamma(Q).$ Thus, $M=N+K_1$ implies $M=N+Q,$ and so $M=N+\gamma(Q).$ Hence, $Q\in \{L\in Spec(M)\mid \gamma(L)+N=M\}=\gamma^{-1}(\emph{m}(N)).$ Then, $\U(K_2)\subseteq \gamma^{-1}(m(N).$
Then, $\gamma^{-1}(m(N)=\U(K_2).$ Therefore, $\gamma$ is continuous.

To conclude this prove, we see that $\Max^{fi}(M)$ is Hausdorff. Let $\mathcal{M},\mathcal{N}\in Max^{fi}(M).$ Considere the following closet sets of $Spec(M),$ $\mathcal{V}(\mathcal{M})=\{\mathcal M\}$ and  $\mathcal{V}(\mathcal{N})=\{\mathcal{N}\}.$ Since $Spec(M)$ is normal, then there exist two disjoint open sets $\mathcal{U}_1$ and $\mathcal{U}_2$ of $Spec(M)$ satisfying $\{\mathcal{M}\}\subseteq \mathcal{U}_1$ and $\{\mathcal{N}\}\subseteq \mathcal{U}_2.$
Thus, $\mathcal{M}\in \mathcal{U}_1\cap \Max^{fi}(M)$ and $\mathcal{N}\in \mathcal{U}_2\cap Max^{fi}(M).$ Consequently, $\Max^{fi}(M)$ is Hausdorff.

\end{proof}

Now, recall that a ring $R$ is said to be a  $pm-$ring if every prime ideal is contained in a unique maximal ideal. In the study of $Spec(R)$ and $\mx(R)$ for a commutative ring,  $pm-$rings have taken an important role, for instance, we have the {\it Demarco-Orsati-Simmons Theorem} which states that,

\begin{thm}\cite{de1971commutative, simmons1980reticulated}
	Let $R$ be a commutative ring. Then:
	$R$ is a pm ring if and only if $\Max(R)$is a retract of $Spec(R)$ if and only if $Spec(R)$ is normal if and only if $R$ is strongly harmonic if and only if $R$ is Gelfand.
\end{thm}

In \cite{sunrings91} is extended the Demarco-Orsati-Simmons Theorem for symmetric rings (which includes the commutative rings). We could not find a good generalization of symmetric rings for modules which be suitable to give a version of the Demarco-Orsati-Simmons Theorem in the module-theoretic context. We finish this paper with a Theorem inspired in the Demarco-Orsati-Simmons Theorem as a compendium of our results. 

As a generalization of pm-rings, in \cite{medina2017attaching} it was introduced the following definition for modules.

\begin{dfn}\cite[Definition 5.5]{medina2017attaching}\label{pmmodule}
An  $R$-module $M$ it is said to be  a \emph{pm-module} if every prime submodule is contained in a unique maximal submodule.  
\end{dfn}

\begin{thm}\label{DOS}
	Let $M$ be projective in $\sm$ such that $\Lambda^{fi}(M)$ is compact and $\mx(M)$ compact. Consider the following conditions 
	\begin{enumerate}[label=\emph{(\alph*)}]
		\item $M$ is a Gelfand module.
		\item $M$ is a quasi-duo strongly harmonic module.
		\item $M$ is a quasi-duo pm-module with $\mx(M)$ Hausdorff.
		\item $M$ is a quasi-duo with $\mx(M)$ is Hausdorff and $\mx(M)$ is a retract of $Spec(M)$.
		\item $M$ is a quasi-duo modulo such that $Spec(M)$ is normal.
		
	\end{enumerate}
	Then the implications \emph{(a)$\Leftrightarrow$(b)$\Rightarrow$(c)$\Rightarrow$(d)$\Rightarrow$(e)} hold. If in addition $0=\bigcap\mx(M)$, all the conditions are equivalent.
\end{thm}

\begin{proof}
	(a)$\Leftrightarrow$(b) It follows from Theorem \ref{SHDuoGelfand}.
	
	(a)$\Rightarrow$(c) From Corollary \ref{gelcoa} and Lemma \ref{Glemma10.11H}, every element of $\Lambda^{fi}(M)$ is contained in a maximal submodule. It follows from Lemma \ref{Gispm} that $M$ is a pm-module.
	
	(c)$\Rightarrow$(d) Since $M$ is a pm module, for every $P\in Spec(M),$ theres exists a unique $\mathfrak{M}_{P}$ maximal submodule containing $P.$ Let $\gamma\colon Spec(M)\to \Max(M)$  defined as $\gamma(P):=\mathfrak{M}_P.$ It is clear that $\gamma(N)=N$ for each $N\in\mx(M).$ Also, notice that $\gamma$ is continuous. Indeed, let $\mathcal{V}({K})\cap \Max(M)=\{N\in\mx(M)\mid K\leq N\}$ be a basic closed set of $\mx(M).$ Then, $\gamma^{-1}(\mathcal{V}({K})\cap \Max(M))=\{P\in Spec(M)\mid \gamma(P)\in \mathcal{V}({K})\cap \Max(M)\}=
	\{P\in Spec (M)\mid  K\leq \mathfrak{M}_p \}\subseteq \mathcal{V}(K).$
Now, let $P\in \mathcal{V}(K).$ Since $M$ is pm and by Lemma \ref{compatomic}, there exists a unique maximal $\mathfrak{M}_P$ such that $P\leq \mathfrak{M}_P.$ Thus, $K\leq \mathfrak{M}_P=\gamma(P),$ and so, $P\in \gamma^{-1}(\mathcal{V}({K})\cap \Max(M)).$ Hence, $\gamma^{-1}(\mathcal{V}({K})\cap \Max(M))$ is a basic open set in $Spec(M).$ Then, $\gamma$ is continuous function.
Therefore, $\gamma$ is a retraction.

(d)$\Rightarrow$(e) It follows from Proposition \ref{retractnormal}.
	
	Assume $0=\bigcap\mx(M)$. Hence	(d)$\Rightarrow$(a) follows from Theorem \ref{theoremsth}.
\end{proof}

\begin{cor}
	Let $R$ be a ring. Consider the following conditions
	\begin{enumerate}[label=\emph{(\alph*)}]
		\item $R$ is a Gelfand ring.
		\item $R$ is a quasi-duo strongly harmonic module.
		\item $R$ is a quasi-duo pm-module with $\mx(R)$ Hausdorff.
		\item $R$ is a quasi-duo with $\mx(R)$ is Hausdorff and $\mx(R)$ is a retract of $Spec(R)$.
		\item $R$ is a quasi-duo modulo such that $Spec(R)$ is normal.
	\end{enumerate}
Then the implications \emph{(a)$\Leftrightarrow$(b)$\Rightarrow$(c)$\Rightarrow$(d)$\Rightarrow$(e)} hold. If the Jacobson radical of $R$ is zero, all the conditions are equivalent.
\end{cor}

\subsection{Questions and possible lines to work out }
Here we leave some questions which we were not able to answer in these paper:

In Proposition \ref{shdirsum} it was proved that fully invariant direct summands of a strongly harmonic module inherit the property. So, we rise the question: 

\vspace{1em}

{\it Q$_1$}: Is the condition of being strongly harmonic module closed under direct summands?

\vspace{1em}

In \cite{PepeyJaime} is proved that,

\begin{thm}  \cite[Theorem 5.10]{PepeyJaime} Let R be a commutative ring and M be a faithful multiplication R-module and $QM \neq M$ for all maximal ideals $Q$ of $R.$ Then the topological spaces $Spec(R)$ and $Spec(M)$ are homeomorphic.
\end{thm}

\noindent Combining this result with Demarco-Orsati-Simmons Theorem and  \cite[Proposition 1.6]{Tuganbaev03},  the following result  is gotten:

\begin{thm}\label{sbtdos}
	Let $R$ be a commutative ring and $M$ a  faitfhul multiplication module satisfying that $IM\neq M$ for every maximal ideal. Then, $M$ is finitely generated, and  the the following conditions are equivalent.
	\begin{enumerate}[label=\emph{(\alph*)}]
		\item $R$ is pm
		\item $Spec(R)$ is normal.
		\item $\mx(R)$ is a retract of $Spec(M)$ and Hausdorff.
		\item $R$ is strongly harmonic.
		\item $M$ is strongly harmonic.
		\item $Spec(M)$ is normal.
		\item $\mx(M)$ is a retract of $Spec(M)$ and Hausdorff.
	\end{enumerate}
\end{thm}

In sight of Proposition \ref{sbtdos}. 

\vspace{1em}

{\it Q$_2$:} Is $M$ a pm-module, if $M$ is a multiplication module and $\mx(M)$ is a retract of $Spec(M)$?

\vspace{1em}

In \cite{sunrings91}, Shu-Hao gave an analogous to Demarco-Orsati-Simmons Theorem for non commutative symmetric rings:

\begin{thm}\cite[Theorem 2.3]{sunrings91}.
	Let $R$ be a weakly symmetric ring. Then the following are equivalent:
\begin{enumerate}[label={(\alph*)}]
	\item $R$ is pm.
	\item $\mx (R)$ is a continuous retract of $Spec (R),$
	\item $Spec (R)$ is normal (not necessary $T_1$),
	and these imply the Hausdorffness of $\mx(R).$ 
\end{enumerate}
\end{thm}

\begin{thm}\cite[Theorem 2.4.]{sunrings91} 
	Let R be a symmetric ring; then R is pm if and only if R is strongly harmonic.
\end{thm}

\vspace{1em} 

{\it Q$_3 $:} Is there an analogous concept of (weakly) symmetric ring for a module? Is there an  analogous for modules of the Shu-Hao's theorems?

\section*{Acknowledgments}

In the latest revisions of this manuscript, we received the sad news that Professor Harold Simmons passed away. For us, Professor Simmons was a mentor. His ideas and theories have been a deep influence in our group of ring and module theory besides his outstanding contributions to point-free topology and mathematical logic. We will miss him.

We are deeply grateful to Professor Harold Simmons for having communicated us the manuscript  \cite{simmonssome}, and also his kind comments when they were asked. Many of the ideas of this investigation came up during a seminar made around that document.

We want to express our gratitude to the referee for her/his comments which improved substantially this manuscript.

This work has been supported by the grant UNAM-DGAPA-PAPIIT IN100517.

\bibliographystyle{amsalpha}

\bibliography{research2}

\scriptsize
 \vskip5mm
\begin{itemize}\item	 \textit{Mauricio Medina B\'arcenas$^{\ast}$, }
Department of Mathematical Sciences, Northern Illinois University, Illinois, Dekalb 60115, USA.
\item\textit{Lorena Morales Callejas$^{\star},$} Facultad de Ciencias, Universidad Nacional Aut\'onoma de M\'exico.
Circuito Exterior, Ciudad Universitaria, 04510, M\'exico, D.F., M\'exico.
\item \textit{Martha Lizbeth Shaid Sandoval Miranda$^{(\ddag)}$ }
Departamento de Matem\'aticas, Universidad Aut\'{o}noma Metropolitana - Unidad Iztapalapa, Av. San Rafael Atlixco 186, Col. Vicentina
Iztapalapa, 09340, M\'exico, D.F., M\'exico.. 
\item \textit{\'Angel Zald\'ivar Corichi$^{\dag}$.}
Departamento de Matem\'aticas, Centro Universitario de Ciencias Exactas e Ingenier\'ias, Universidad de Guadalajara, Blvd. Marcelino Garc\'ia Barrag\'an, 44430, Guadalajara, Jalisco, M\'exico.\vspace{5pt}
\end{itemize}
\begin{itemize}
\item[\hspace{33pt}] $\ast$ mmedina6@niu.edu 
\item[\hspace{33pt}] $\star$ lore.m@ciencias.unam.mx
\item[\hspace{33pt}] $\ddag$ marlisha@xanum.uam.mx    (Corresponding author)
\item[\hspace{33pt}]  $\dag$ luis.zaldivar@academicos.udg.mx 
\end{itemize}
\end{document}